\documentclass[11pt]{amsart}
\usepackage[utf8]{inputenc}
\usepackage[english]{babel}
\usepackage[T1]{fontenc}
\usepackage{indentfirst,microtype}
\usepackage[margin=2.3cm]{geometry}
\usepackage{amsfonts,amsthm,amsmath,amssymb,amscd,cite}
\usepackage{enumerate}
\usepackage[colorlinks=true,linkcolor=blue,citecolor=magenta]{hyperref}
\usepackage{eurosym}
\usepackage{xcolor}
\usepackage[T1]{fontenc} \usepackage{lmodern}
\usepackage{verbatim}
\usepackage{amsmath}
\usepackage{tikz}
\usepackage{tikz-cd} 
\usepackage{graphicx}

\newcommand{\R}{\mathbb{R}}

\newcommand{\N}{\mathbb{N}}

\newcommand{\Sing}{\mathrm{Sing}}

\newcommand{\mesh}{\mathrm{mesh}}

\newtheorem{thm}{Theorem}[section]

\newtheorem*{thm*}{Theorem.}

\newtheorem*{thma}{Theorem A}

\newtheorem*{thmb}{Theorem B}
\newtheorem*{thmc}{Theorem C}

\newtheorem{proposition}{Proposition}[subsection]

\newtheorem{cor}{Corollary}[subsection]
\newtheorem{lemma}{Lemma}[subsection]

\newtheorem{definition}{Definition}[subsection]

\newtheorem{remark}{Remark}[subsection]
\newcommand{\kei}[1]{{k(I)}}

\newcommand{\I}{\mathcal{I}}

\newcommand{\norm}[1]{\Vert #1 \Vert}

\newcommand{\D}{\mathrm{D}}

\newcommand{\be}{\begin{equation}}
\newcommand{\ee}{\end{equation}}
\newcommand{\bes}{\begin{equation*}}
\newcommand{\ees}{\end{equation*}}

\title[$\mathcal{C}^{1+\alpha}$-regularity of conjugacies of linearizable GIETs]{Regularity of conjugacies of linearizable generalized interval exchange transformations}

\author{Selim Ghazouani}
\author{Corinna Ulcigrai}

\begin{document}

\begin{abstract}
 We consider generalized interval exchange transformations (GIETs) of $d\geq 2$ intervals which are \emph{linearizable}, i.e.~differentiably conjugated to standard interval exchange maps (IETs) via a diffeomorphism $h$ of $[0,1]$ and study the regularity of the conjugacy $h$. Using a renormalisation operator obtained accelerating Rauzy-Veech induction,  we show that, under a full measure condition on the IET obtained by linearization, if the orbit of the GIET under renormalisation converges exponentially fast in a $\mathcal{C}^2$ distance to the subspace of IETs, there exists an exponent $0<\alpha<1$ such that $h$ is $\mathcal{C}^{1+\alpha}$. Combined with the results proved by the authors in \cite{SU}, this implies  in particular the following improvement of the rigidity result in genus two proved in \cite{SU} (from  $\mathcal{C}^1$  to  $\mathcal{C}^{1+\alpha}$ rigidity): 
for almost every irreducible IET $T_0 $ with $d=4$ or $d=5$, for any GIET which is topologically conjugate to $T_0$ via a homeomorphism $h$ and has vanishing boundary, the topological conjugacy $h$ is actually a $\mathcal{C}^{1+\alpha}$ diffeomorphism, i.e.~a diffeomorphism $h$ with derivative $Dh$ which is $\alpha$-H{\"older} continuous.
\end{abstract}


\maketitle 

\tableofcontents

\section{Introduction and main results}

\subsection{Linearization  of GIETs  and rigidity}
We pursue in this article the investigation of the regularity of conjugating maps between smooth generalised interval exchange transformations (GIETs). Generalised interval exchange transformations appear naturally as first-return maps of flows on surfaces, and are thus seen as natural generalisations of circle diffeomorphisms to higher genus. The study of circle diffeomorphisms is a classical topic in dynamical systems, initiated by Poincar{\'e}'s invention of the rotation number followed by Denjoy's important distortion estimates and Arnol'd's introduction of KAM methods to the topic. The theory culminated with Herman's spectacular treaty \cite{Herman} establishing (amongst other things) the regularity of the map conjugating most minimal circle diffeomorphisms to their linear model. 

 Efforts to extend these results to the higher genus case (and thus GIETs) have been ongoing since the early Eighties, 
see in particular the seminal works by Forni \cite{Fo:ann} and Marmi, Moussa and Yoccoz \cite{MMY, MMY2, MMY3};
 we refer the reader to the article \cite{SU} for more references and a detailed discussion about rigidity questions for GIETs. 

 In \cite{SU}, it is proven that under a generic arithmetic condition, genus $2$ minimal GIETs with vanishing boundary are $\mathcal{C}^1$-conjugate to their linear model. The proof follows a general theme in one-dimensional dynamics: we show that  the orbits of such GIETs under a suitable renormalisation operator converge (at an exponential rate) to their linear model and derive the regularity of the conjugacy from this fact. In many other places in one-dimensional dynamics it is shown that \textit{exponential convergence} of renormalisation actually imply that the conjugacy is of class $\mathcal{C}^{1+ \alpha}$. In the present article we extend this implication to the case of GIETs (see Theorem \ref{thmb} for a precise statement), thus improving upon the main result of \cite{SU}.  
\vspace{2mm}


\subsection{Regularity of conjugacies and foliations in genus two}\label{sec:GIETrigidityintro}
 Let us denote denote by $\mathcal{I}_d$, for a fixed $d\geq 2$, the space of standard irreducible interval exchange transformations with $d$ branches (see \S~\ref{giet} for the definition of irreducible). The space $\mathcal{I}_d$  carries a natural Lebesgue measure (see \S~\ref{giet}). 
We prove the following rigidity result, which improves on the rigidity result previously proved in \cite{SU}. 
\begin{thma}[$\mathcal{C}^{1+\alpha}$-rigidity of GIETs with $d=4$ or $d=5$]\label{thma}
Let $d=4$ or $5$. For Lebesgue almost every\footnote{{Here \label{Lebfootnote} the measure is the Lebesgue measure on the parameter \emph{standard} IETs $\mathcal{I}_d$, i.e.~a result holds for a full measure set of IETs in $\mathcal{I}_d$, if it holds for all \emph{irreducible} combinatorial data and Lebesgue-almost every choice of \emph{lengths} of the continuity intervals. See \S~\ref{giet} for details.}}  interval exchange transformation $T_0$ in $\mathcal{I}_4 \cup \mathcal{I}_5$ the following holds. Given   $T$ any  $\mathcal{C}^3$-generalized interval exchange map $T$ whose boundary $\mathcal{B}(T)$ vanishes, if $T$ is topologically conjugate to $T_0 $, then there exists  $0<\alpha = \alpha(T_0)<1$ such that  the conjugacy between $T$ and $T_0$ is actually a diffeomorphism of $[0,1]$ of class $\mathcal{C}^{1+\alpha}$. 
\end{thma}
The existence of a  diffeomorphism $h $ of $[0,1]$ of class $\mathcal{C}^{1}$ which conjugates  $T $ and $T_0$ under the same assumptions of the theorem was one of the main results proved by the authors in \cite{SU}. The novel part of this result is that $h$ is actually $\mathcal{C}^{1+\alpha}$.


\smallskip 
\noindent{\it Optimal regularity.} 
Contrary to the theory of circle diffeomorphisms (where a $\mathcal{C}^{\infty}$ circle diffeo which is topologically conjugate to a rotation, is actually smoothly conjugate by a $\mathcal{C}^{\infty}$ conjugacy)  
here the conjugacy is expected to typically fail\footnote{This corresponds to Forni's and Marmi-Moussa-Yoccoz  non-trivial obstructions to solving the cohomological equation: Marmi-Moussa-Yoccoz have indeed shown that asking for more regular conjugacy forces GIETs to live in positive codimension submanifolds of the $\mathcal{C}^1$-conjugacy class; the codimension is an exact reflection of the aforementioned obstruction. } to be $\mathcal{C}^2$, so this result is expected to be optimal (for a full measure set of IETs).  
This result indicates therefore that GIETs are closer (as far as the regularity of the conjugating map is concerned)  to essentially non-linear rigid dynamical systems (such as unimodal maps and circle map with breaks or critical points) for which the conjugacy is typically not $\mathcal{C}^2$ but $\mathcal{C}^{1+\alpha}$ for some  $0<\alpha<1$.

\smallskip 
\noindent{\it The exponent $\alpha$.} It turns out that the exponent $\alpha(T_0) > 0$ can be shown to be independent of $T_0$ for $T_0$ in a set of full measure. It is somewhat obvious from the proof: as many results of this kind, the exponent $\alpha$ depends only upon the exponential speed of convergence of renormalisation. The optimal value of $\alpha$ for GIETs remains completely open.

\smallskip 
\noindent{\it The boundary invariant.}
The \emph{boundary} $\mathcal{B}(T)$ of $T$ in the statement is a $\mathcal{C}^1$-conjugacy invariant associated to a GIET $T$ (for the definition of  $\mathcal{B}$, which is based on Marmi-Moussa-Yoccoz \emph{boundary} operator from \cite{MMY3, MarmiYoccoz}, see \cite{SU}). Requiring that $\mathcal{B}(T)$ vanish is a necessary condition:
 two GIET that are topologically conjugate but have different boundaries cannot be differentiably conjugate, simply because the boundary is $\mathcal{C}^1$-conjugacy invariant. We note, for the reader who is familiar with the one-dimensional dynamics literature, that the assumption that $\mathcal{B}(T)$ be equal to zero, in the special case where $T$ is a circle maps with breaks, reduces to the assumption that the \emph{non-linearity} $\eta_T$ (see \S~\ref{sec:nl}) has integral zero and that the special pair $(T_1,T_2)$, where $T_1,T_2$ are the two branches of $T$, corresponds to a diffeomorphism without break points (see \cite{SU} for details). 

Geometrically, when $T$ is the Poincar{\'e} map of a minimal foliation on a surface, the boundary $\mathcal{B}(T)$ encodes the $\mathcal{C}^1$-holonomy around the saddles of the foliation (see \cite{SU}). The assumption that $\mathcal{B}(T)$ be zero is equivalent to {asking that the corresponding foliation have trivial $\mathcal{C}^1$-holomony around the singularities (see \cite{SU} for details). 
Using this observation, 
 one can deduce from Theorem~\ref{thma} (as in \cite{SU}, see in particular \S~6.4.3) the following consequence for foliations on surfaces of genus two.} 

\begin{cor}[Foliations $\mathcal{C}^{1+\alpha}$-rigidity in genus two]\label{cor:foliationsrigidity}
Let $S$ be a closed orientable surface of genus $2$. There exists  a full measure set $\mathcal{D}$ of orientable measured foliations on $S$ such that, if 
 $\mathcal{F}$ a  is a {minimal} {orientable} foliation on $S$ of class $\mathcal{C}^3$ such that:
\begin{itemize}
\item[(i)] $\mathcal{F}$ is topologically conjugate to a measured foliation $\mathcal{F}_0$ beloning to the full measure set $\mathcal{D}$;
\item[(ii)] the $\mathcal{C}^1$-holonomies of $\mathcal{F}$ at all singularities  vanish;
\end{itemize}
\noindent then there exists $0<\alpha<1$ such that $\mathcal{F}$ is actually $\mathcal{C}^{1+\alpha}$-conjugate to $\mathcal{F}_0$.
\end{cor}

The notion of full measure 
on orientable measured foliations (with fixed type of singularities) used in Corollary~\ref{cor:foliationsrigidity} is given by the so-called \emph{Katok fundamental class},  see \S~\ref{sec:final}.


\subsection{Regularity  from exponential convergence of renormalisation}\label{sec:renormalisationintro}
 The proof of the $\mathcal{C}^{1+\alpha}$-regularity of the conjugacy in the Main Theorem (as well as the existence of a $\mathcal{C}^{1}$-conjugacy, proved in \cite{SU}) is based upon renormalisation techniques.  In terms of renormalisation, we prove a more general result, which holds for any $d\geq 2$ and guarantees the $\mathcal{C}^{1+\alpha}$-regularity of the conjugacy of a GIET $T$ to its linear model as long as the orbit of $T$ under renormalisation converges exponentially fast to the subspace of IETs (see Theorem~\ref{thmb}), as we now explain.
 
 \smallskip
 Let $\mathcal{X}^r_d$ denote the space of all GIETs of class $\mathcal{C}^r$ on $d$ intervals  with an irreducible combinatorics (see  \S~\ref{giet} for definitions). 
is  obtained associating to a  
 given a GIET $T:I\to I$ on $I=[0,1]$ in the domain of $\mathcal{R}$ 
 another GIET, which we will call $\mathcal{R}(T)$, 
which is obtained by suitably choosing an  subinterval $I'\subset I$ (so that the induced map is well defined and is again a GIET  $T'$ of the same number of intervals) and considering the \emph{induced map} of $T$ and \emph{normalizing}, i.~.e.~conjugating by the affine transformation which maps $I'$ to $I$, so that the image is again a GIET on $I$.  The renormalisation operator that we study is 
 an acceleration of  Rauzy-Veech induction, a classical algorithm first introduced  by Rauzy \cite{Ra:ech}  and  by Veech \cite{Ve:gau, Ve:inI} to \emph{renormalize} standard IET and study their fine ergodic properties. This renormalisation can be defined also on GIETs with no connections and plays a crucial role also in the study of GIETs, especially those who are conjugate to a standard IET (see e.g.~in \cite{MMY} and \cite{SU}). 
 
 \smallskip

   The general statement about regularity of the conjugacy is the following result, which is valid  for any $d\geq 2$ but  conditional to the assumption of exponential convergence of the renormalisation dynamics.

\begin{thmb}[[exponential convergence gives a.s.~$\mathcal{C}^{1+\alpha}$-conjugacy]
\label{thmb}
For any $d\geq 2$, for a.e.~IET $T_0$ in $\mathcal{X}^r_d$ the following holds. Assume that $T$ is a GIET in  $\mathcal{X}^r_d$, $r\geq 3$, which is  conjugated to $T_0$ by a diffeomorphism of $[0,1]$ of class $\mathcal{C}^{1}$. Then if the orbit  $(\mathcal{R}^m(T))_{m\in\mathbb{N}}$ of $T$ under renormalisation converges exponentially fast, in the $\mathcal{C}^2$ distance, to the subspace $\mathcal{I}_d$ of (standard) IETs, i.e.
$$d_{\mathcal{C}^2} (\mathcal{R}^m(T),  \mathcal{I}_d)\leq C \lambda^m, \qquad \, \forall m \in \N,
$$
for some $C>0$ and $0<\lambda<1$ (where the distance  $d_{\mathcal{C}^2} $  is defined in \S~\ref{giet}), then there exists $0<\alpha<1$  (depending only on $\lambda$) such  that the conjugacy between $T$ and $T_0$ is actually a diffeomorphism of $[0,1]$ of class $\mathcal{C}^{1+\alpha}$.
\end{thmb}
\noindent 

The proof of Theorem \ref{thmb} (which is given in \S~\ref{sec:proof})  constitutes the heart of this paper. We comment in \S~\ref{sec:diff} on both similarities and difficulties in deducing $\mathcal{C}^{1+\alpha}$ from exponential convergence of renormalization in this setting compared to related results in the literature.

The full measure set of IETs in the statement of Theorem~\ref{thmb} is explicitely characterized by a simple \emph{Diophantine-like} condition (see the survey \cite{Ul:ICM} on the notion of Diophantine-like conditions for IETs), namely a condition expressed in terms of growth conditions of the matrices of (an acceleration of) the Rauzy-Veech incidence matrices (as defined in \S~\ref{sec:renormalisation}). We refer the reader to Definition~\ref{def:DC} for the condition. 

\subsection{Convergence of renormalisation in genus two}\label{sec:renormalisationg2}
A crucial difference between circle diffeomorphisms and GIETs, though, is that even \emph{convergence} of renormalisation (namely that $ d_{\mathcal{C}^1} (\mathcal{R}^m(T),  \mathcal{I}_d)\to 0$ as $m\to \infty$) is \emph{rare}. In particular, the orbit  $(\mathcal{R}^m(T))_{m\in\mathbb{N}}$ of  a GIET $T$ under renormalisation often diverges (even though ina  controlled way, see \cite{SU} for a dynamical dichotomy which characterizes the way in which divergence happens): the set of GIETs for which there is (exponential) convergence of renormalisation are expected to form a lower dimensional subvariety (see the conjectures by Marmi, Moussa and Yoccoz in \cite{MMY3} as well as the result \cite{Selim:loc} by the first author in a special case).

Nevertheless, we showed in \cite{SU} that in genus two (i.e.~for $d=4$ or $d=5$) the existence of a topological conjugacy between $T$ and a standard IET $T_0$ in a suitable full measure subset of $\mathcal{I}_d$  is sufficient to guarantee exponential convergence of renormalisation (under the necessary assumption that the boundary of $T$ vanishes). The following result was indeed proved in \cite{SU}:

\begin{thmc}[from \cite{SU}, rigidity and exponential convergence of renormalisation in genus two]\label{thmc}
Assume that $r\geq 3$ and let $d=4$ or $5$. There exists $0<\rho<1$ such that, for Lebesgue almost every  interval exchange transformation $T$ in $\mathcal{I}_4 \cup \mathcal{I}_5$, given   $T$ any  $\mathcal{C}^3$-generalized interval exchange map $T$ whose boundary $\mathcal{B}(T)$ vanishes, if $T$ is \emph{topologically conjugate} to $T_0 $, then  orbit  $(\mathcal{R}^m(T))_{m\in\mathbb{N}}$ of $T$ under renormalisation converges exponentially fast, in the $\mathcal{C}^2$ distance, to the subspace $\mathcal{I}_d$ of (standard) IETs, i.e.
\begin{equation}
\label{exp:conv}
d_{\mathcal{C}^2} (\mathcal{R}^m(T),  \mathcal{I}_d)\leq C \rho^m, \qquad \, \forall m \in \N,
\end{equation}
Furthermore, in this case $T$ and $T_0$ are \emph{differentiably} conjugate, i.e.~the conjugacy is a $\mathcal{C}^{1}$ diffeomorphism of $[0,1]$.
\end{thmc}
To show this result, in \cite{SU} it is shown first that the existence of a topological conjugacy in genus two  prevents renormalisation to diverge (this part exploits a result proved  by Marmi, Moussa and Yoccoz in \cite{MMY2} on existence of wandering intervals in affine GIETs). In light of the dynamical dichotomy proved in   \cite{SU}), it  then follows that there is convergence of renormalisation and, under a full measure condition on the IET (which plays the role of rotation number), that this convergence happens at exponential speed. The differentiability of the topological conjugacy can then be proved from exponential convergence of renormalisation generalizing methods which go back to the seminal work of Michel Herman on linearization of circle diffeomorphisms.

The combination of Theorem~C and Theorem~B yields immediately Theorem~A, see  \S~\ref{sec:final}.  
We stress that Theorem~B requires no assumption on $d\geq 2$. Theorem~C is also expected to hold for any $d\geq 2$. A great part of the results in \cite{SU} are already proved for any $d\geq 2$; the restriction $d=4,5$ in Theorem~C comes from the use of Marmi, Moussa and Yoccoz  work \cite{MMY2} (which requires a technical assumption which is automatic in genus two). Provided that a generalization of this result will be proved, Theorem~A will automatically hold true for any $d\geq 2$.


\subsection{On the proof strategy and tools and related results}  \label{sec:diff}
In  one-dimensional, infinitely renormalisable dynamical systems, it is expected that two maps whose renormalisations are getting asymptotically close at an exponential rate must be $\mathcal{C}^{1+ \alpha}$-conjugate for some $\alpha > 0$ (provided some mild arithmetic condition is satisfied). The proof of this important technical step (moving from hyperbolicity of renormalisation to $\mathcal{C}^{1+ \alpha}$-rigidity) can prove difficult, as it requires a careful comparison of the dynamical partitions induced by the infinite renormalisation. Different methods were exploited in different setting and can be found at play for example in the following 
references (see also the references therein):
\begin{itemize}
\item \cite{McMullen2}[Chapter 9] for quadratic unimodal maps; 
\item \cite{deFariadeMelo} for (bounded type) critical circle maps;
\item \cite{KhaninKhmelev} for (bounded type) circle maps with break points;
\item \cite{MartensPalmisano} for some critical Lorenz maps.
\end{itemize}
\noindent A way around having to look directly into the geometry of the associated dynamical partitions exists when \textit{one of the two conjugate maps is linear} (for instance in the case of circle diffeomorphisms or of GIETs with vanishing boundary). In this particular case,  the existence of a $\mathcal{C}^1$-conjugacy (respectively $\mathcal{C}^{1+\alpha}$-conjugacy) of a map $T$ to a linear model $T_0$ is equivalent to the observable $\log \mathrm{D}T$  being a $\mathcal{C}^0$ (respectively $\mathcal{C}^{\alpha}$) co-boundary,  namely to the existence of a  $\mathcal{C}^0$ (respectively $\mathcal{C}^{\alpha}$) solution $\varphi$ to the \textit{cohomological equation} equation
\begin{equation}\label{eq:coheq}
\varphi\circ T  - \varphi = \log DT.
\end{equation}     
\noindent It is then possible (at least theoretically) to use results and methods about solving the \eqref{eq:coheq}  to establish $\mathcal{C}^{1+ \alpha}$-rigidity from exponential convergence of renormalisation, sparing one a delicate analysis of the geometry of the respective dynamical partitions. 

This approach is pursued by Khanin and Templisky in \cite{KT:Her} (see also the previous work \cite{SK:Her} by Khanin and Sinai) for the case of circle diffeomorphisms (to reprove Herman's theory in low regularity). The methods of \cite{KT:Her} do not generalise straightforwardly to the case of GIETs,  essentially because of some of the technical difficulties introduced by  the presence of discontinuities.  
 For example, while for
that for circle diffeomorphism  any point can be chosen as  a base point for the renormalization scheme,  GIETs do not enjoy this form of homogeneity; furthemore, dynamical partitions for GIETs have a much less rigid and less clearly understood structure than the corresponding partitions for circle diffeomorphisms.  

The  article \cite{MarmiYoccoz} by Marmi and Yoccoz studies the  regularity of the solutions to the cohomological equation  $\varphi \circ T_0 - \varphi(x) = f$ where $T_0$ is  a \emph{standard} IET  (whose existence under suitable conditions was established in \cite{MMY}) and shows that  under a full measure Diophantine-type condition, when a continuous solution exist for a given $\mathcal{C}^{1+\alpha}$-observable, it is actually  H{\"older} continuous. They approach to study regularity exploits as technical tool what they call \emph{spatial decompositions}. 
Their methods fall short from being applicable to our setting:   since we are assuming the existence of a conjugacy $h$ between $T$ and its linear model $T_0$, one can conjugate the equation \eqref{eq:coheq} to reduce from our setting to the study of the regularity of solutions to the cohomological equation for $T_0$ for $f:= \log \mathrm{D}T\circ h$, but since $h$ is \textit{a priori} is only $\mathcal{C}^1$, the regularity of the observable $f$  is a priori only $\mathcal{C}^1$, so the results in \cite{MarmiYoccoz} cannot be applied.

\smallskip
The key technical contributions of this article are two-fold: one one hand \emph{analytic} (controlling the effect of non-linear terms to the geometry of partition elements), on the other \emph{combinatorial} (introducing an alternative approach to control the combinatorial structure of dynamical partitions, which is somehow \emph{hybrid} between the spatial decompositions introduced by Marmi and Yoccoz in \cite{MarmiYoccoz} and the classical approach for circle diffeomorphisms pursued in \cite{KT:Her}). 

\smallskip
\noindent \emph{New analytic tools.}  The analytic crucial step in the proof is to show that from the assumption of exponential convergence of renormalisation, one can gain extra analytical information on the observable $\log \mathrm{D}T$, which in turn can be used to control of the distorsion of floors of towers in the dynamical partition (this control can be deduced from Proposition~\ref{prop:reminderestimate}, which provides an estimate of what we call \emph{broken Birkhoff sums}, see Definition~\ref{def:brokenBS}).  
We remark that Proposition~\ref{prop:reminderestimate} provides a non-linear counterpart for Lemma 3.20 in the linear setting of \cite{MarmiYoccoz} and should in principle suffice to apply the approach of \cite{MarmiYoccoz} to our setting, namely to the observable $\log \mathrm{D}T$.  We provide instead also an alternative approach to the combinatorial part of \cite{MarmiYoccoz} (which is quite involved). 

\smallskip
\noindent \emph{New combinatorial tools.}   In order to investigate the regularity of the conjugacy,  a combinatorial understanding of the structure of dynamical partitions and orbits is needed. Exponential convergence of renormalization controls rather directly the convergence of Birkhoff sums of the function $\log \mathrm{D}T$ at special times given by renormalization (namely so-called \emph{special Birkhoff sums}, see~\ref{eq:sBS} and Lemma~\ref{SBSviaR}). Marmi and Yoccoz in \cite{MarmiYoccoz} analyse the spatial  variation  of a solution $\varphi$ of the cohomological equation \eqref{eq:coheq} for the function $f:= \log \mathrm{D}T$ using a \emph{spatial decomposition} (into blocks of the form  $\Delta \varphi (J) := \varphi(b)-\varphi(a)$ where $J=(a,b)$ is a floor of a dynamical partition); one then has to related the control of the quantities $\Delta \varphi (J)$ to special Birkohff sums.  Rather  than spatial decompositions as in \cite{MarmiYoccoz}, we  use \emph{time-decomposition} of Birkhoff sums of the function $f=\log DT$ along the orbit $\mathcal{O}_T(0)$ of the point $x_0=0$ (as in the classical case of circle diffeomorphisms, see e.g.~\cite{KT:Her}). Due to the lack of \emph{homogeneity}, though, we cannot choose $x_0$ so we approximate both $x$ and $y$  through points in $\mathcal{O}_T(0)$,  by building what which we call \emph{single orbit approximations}, see \S~\ref{sec:appr} 
(in particular Propostion~\ref{prop:appr}) for details. We believe that this new approach is of independent interest and may find further applications.

\section{Background material}\label{background}
In this preliminary section we recall some basic definitions and background on generalized interval exchange maps (in \S~\ref{giet}) and their renormalisation (see \S~\ref{sec:renormalisation}), as well as a few non-linear tools (see \S~\ref{sec:nl}). We assume that the reader has familiarity of basic definitions and properties of Rauzy-Veech induction, on which many excellent lecture notes are available (see e.g.~\cite{Vi:IET} or \cite{Yoc:Clay}). An understanding of structure of dynamical partitions and  Rohlin towers induced by these type of induction is especially crucial (see \S~\ref{sec:renormalisation} below, or also \cite{Ul:abs}, \S~2.2). The reader interested in  cohomological methods to solve the conjugacy problem can find a general introduction in \cite{KH:mod} (Chapter 12) or \cite{Si:Top} (Lecture $10$) and  examples in different settings in \cite{SK:Her, MMY, SU}.

\subsection{Generalized interval exchange transformations}\label{giet} 
Let us start by recalling the  definition of generalized interval exchange transformations, or, for short, GIETs. 
Let $d\geq 2$ be an integer and $r$ a positive real number. A $\mathcal{C}^r$-generalized interval exchange transformation (GIET) of $d$ intervals, or for short a $d$-GIET of class $r$, is a map $T$ from the interval $[0,1]$ to itself such that:
\begin{itemize}
\item[(i)] there are two partitions (up to finitely many points) of $[0,1] = \bigcup_{i=1}^d{I^t_i} =  \bigcup_{i=1}^d{I^b_i} $  of $[0,1]$ into  $d$ open disjoint subintervals, called the \emph{top} and \emph{bottom} partition; the subintervals are denoted respectively $I_i^t$, for $1\leq i \leq d$, and $I_i^b$, for $1\leq i \leq d$;
\smallskip

\item[(ii)]  for each $1\leq i\leq d$, $T$ restricted to $I_i^t$ is an orientation preserving diffeomorphism onto $I_{i}^b$ of class $\mathcal{C}^r$;
\smallskip

\item[(iii)] $T$ extends to the closure of $I_i^t$ to a $\mathcal{C}^r$-diffeomorphism onto the closure of $I_{i}^b= T(I_i^t)$.

\end{itemize}
See Figure \ref{IET} (left) for an example of a graph of a GIET with $d=4$. 
We will call the restriction $T_i:=T|I^t_i$ of $T$ onto $I^t_i$, for $1\leq i\leq d$, a \emph{branch} of $T$. 

	\begin{figure}
\centering
\begin{minipage}{.5\textwidth}
  \centering
			\def\svgwidth{ 0.8\columnwidth}
 \includegraphics[width=.55\textwidth]{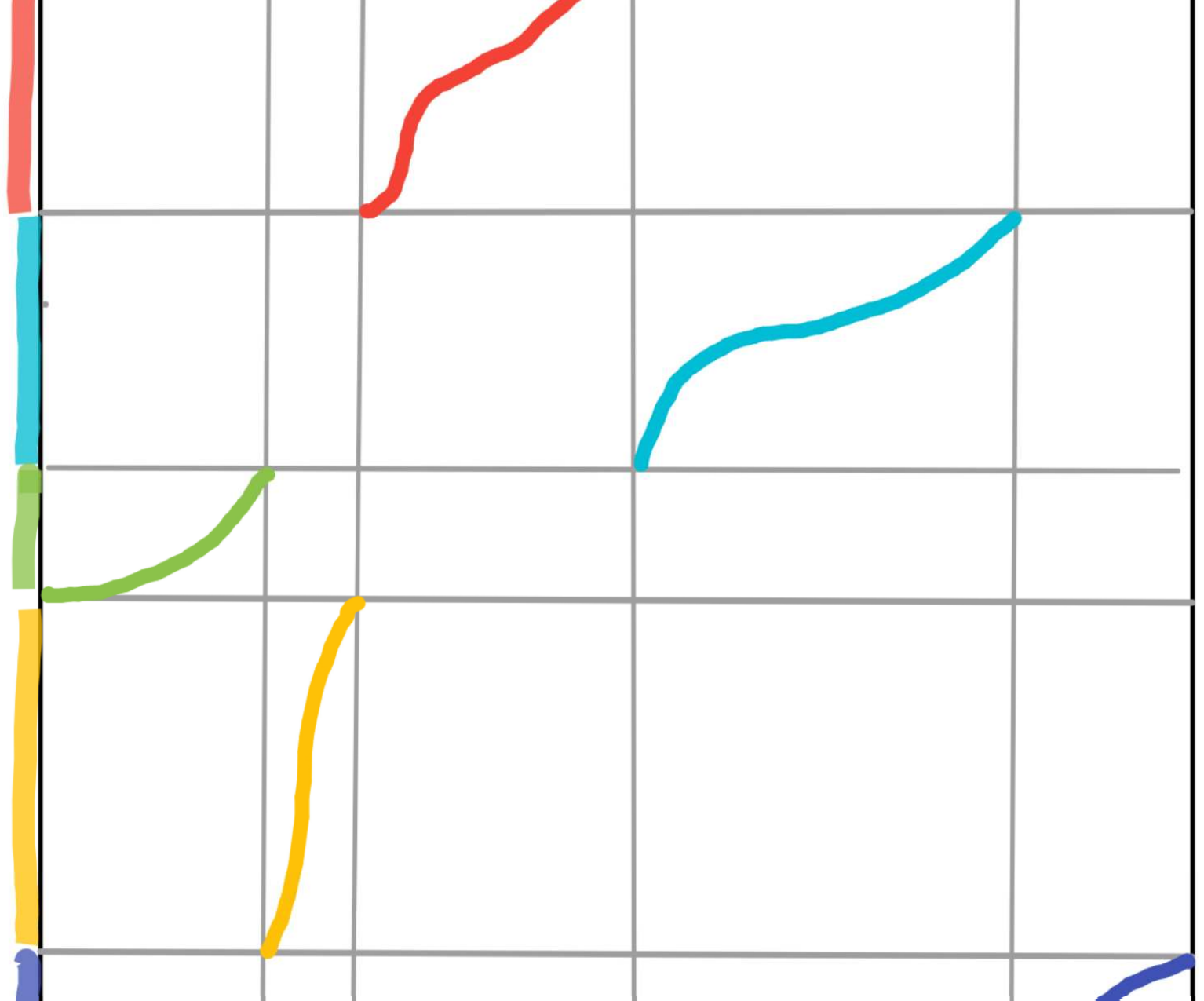}
\end{minipage}%
\begin{minipage}{.5\textwidth}
  \centering
			\def\svgwidth{ 0.8\columnwidth}
			\includegraphics[width=.69\textwidth]{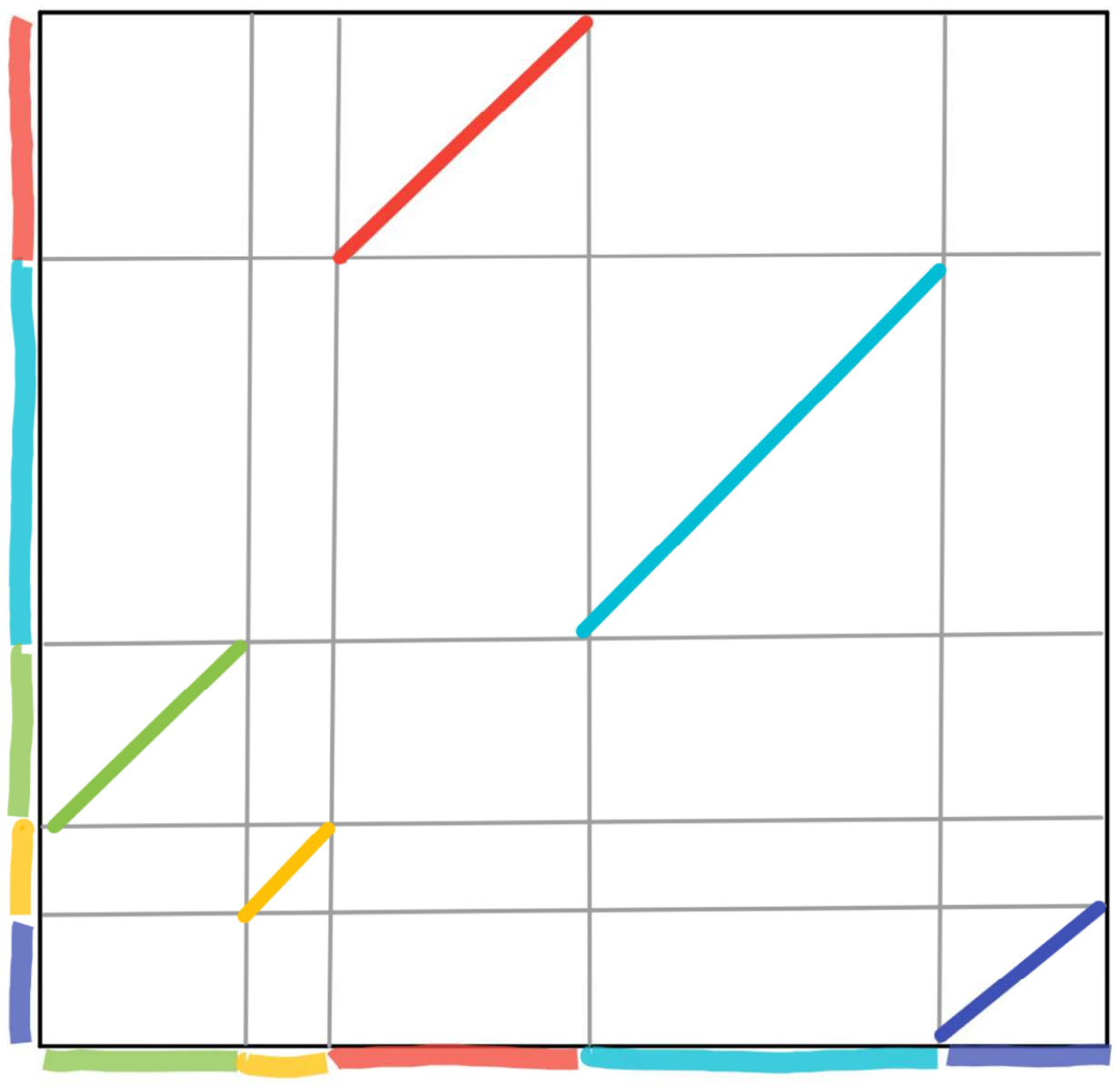}

\end{minipage}
\caption{A generalized  IET (GIET)  and a (standard) IET  with $d=5$.\label{IET}}
\end{figure}

\smallskip
\noindent \textit{Standard interval exchange transformations} (IETs) can be seen a a special cases of generalized interval exchange transformations: 
a GIET $T$ is an (\emph{standard}) \emph{interval exchange transformation} or a $IET$ if $|I^t_i|=|I^b_t|$ for every $1\leq i\leq d$ and the branches $T_i$ of the map $T$, for every $1\leq i\leq d$, are assumed to be \emph{translations}, i.e.~of the form $x\to x+\delta_i$ for some $\delta_i\in\mathbb{R}$. 
See Figure \ref{IET} (right) for an example of a graph of an IET with $d=4$. 

\smallskip
\noindent {\it Conjugacies.} We say that a GIET $T$ is \emph{linearizable} if it is \emph{topologically conjugated} to a standard IET $T_0$, i.e.~there exists a homeomorphism $h:[0,1]\to [0,1]$, called \emph{conjugacy}, such that $h\circ T_1= T_0\circ h$. We say that $T$ is \emph{differentiably linearizable} if $T$ and $T_0$ are \emph{differentiably} conjugate, i.e.~$h$ is a $\mathcal{C}^1$ diffeomorphism of $[0,1]$.

\smallskip
\noindent {\it Relation with foliations.}
We recall that generalized interval exchange transformations appear naturally as Poincar{\'e} first return maps of  orientable foliations on a surface on transversal segments. The discontinuities arise indeed from points on the interval which hit a singularity of the foliation (or an endpoint of the transversal interval) and therefore do not return to the transversal, while the intervals $I_j^t$ are continuity intervals of the Poincar{\'e} map. The smoothness $r$ of the branches depends on the regularity of the foliation. When the foliation is a \emph{measured foliation}, one can choose coordinates so that the Poincar{\'e} map is a standard IET (see e.g.~\cite{Yoc:Clay} or \cite{Vi:IET}). 


\smallskip
\noindent {\it Combinatorial data and irreducibility.}  
The \emph{order} of the intervals (from left to right) at the top and bottom partition of a GIET can be encoded using \emph{two permutations}
 $\pi_t$ and $\pi_b$ of $\{1,\dots, d\})$: $\pi_t$ (resp.~$\pi_b$) describes the order of the intervals in the top (resp.~bottom) partition. 
We call the pair $\pi=(\pi^t,\pi^b) $ the \emph{combinatorial datum} of $T$. We will always assume that the combinatorial datum is \emph{irreducible}, i.e.~for every $1\leq k<d$ we have
$
{ \pi_t \{ 1,\dots, k\} \neq  \pi_b \{ 1,\dots, k\}}.
$
We will denote by $\mathfrak{S}_d^0$ the set of irreducible combinatorial data $\pi=(\pi_t,\pi_b)$ with $d$ symbols. 

\smallskip
\noindent {\it Keane condition.}
We denote by $u_i^t$, for $0\leq i\leq d$ the endpoints of the top partition intervals and, respectively by  $u_i^b$, $0\leq i\leq d$, the endpoints of the bottom partition, in their natural order. 
%
A \emph{connection} is  a triple $(u^t_j, u^b_i, m)$ where $m$ is a positive integer such that $T^m (u^b_j)=v^t_i$. When $T$ is the Poincar{\'e} map of a transveral a flow along the leaves of a foliation on $S$, connections correspond to \emph{saddle connections} on $S$, i.e.~trajectories of the flow which connect two singularities.  
We say that $T$ satisfies the Keane condition (or, simply, that $T$ is \emph{Keane}) if it has \emph{no connections}, i.e.~if no such triple exists. Let us recall that almost every IET in $\mathcal{I}_d$ is Keane and that, as shown by Keane in \cite{Ke:int} (see also \cite{Vi:IET} or \cite{Yoc:Clay}) if a (standard) IET is Keane, then it is minimal. 

\smallskip
\noindent {\it Parameter spaces.} 
For a fixed differentiability class $r \in \R_+$ and number of intervals $d\geq 2$, we define the space $\mathcal{X}^r $ 
 of generalized interval exchange transformations of class $\mathcal{C}^r$ with $d$ intervals, namely $\mathcal{X}^r := \bigcup_{\pi \in \mathfrak{S}_d^0}{\mathcal{X}_{\pi}^r }$ where $\mathcal{X}_{\pi}^r$ is the set of {d-GIET of class} $\mathcal{C}^r $ with associated permation $\pi$. 
The subspace of \emph{(standard)} interval exchange transformations 
 with combinatorics $\pi$ will be denoted by $\mathcal{I}_\pi$. 
For any $d\geq 2$, let us set
$ \mathcal{I}_d := \bigcup_{\pi \in \mathfrak{S}_d}{\mathcal{I}_{\pi} }\subset \mathcal{X}_{d}^r$ (for any $r>0$). 
 
\smallskip
\noindent {\it Profile and coordinates.} 
Given a GIET with combinatorics $\pi$ and continuity intervals $(I_i^t)_{1 \leq i\leq d}$, let us define the \emph{profile} of $T$ to be the vector $\varphi_T:=(\varphi_T^1, \cdots, \varphi_T^d)\in (\mathrm{Diff}^r([0,1]))^d $ whose entries $\varphi_T^j\in \mathrm{Diff}^r([0,1])$ are \emph{renormalized} copies of each branch $T_j$, namely $\varphi_T^i:= a_i \circ  T_i \circ b_i $ where $b_i$ is the unique orientation preserving affine map mapping $I_i^t$ onto $[0,1]$ and  $a_i$ is the unique orientation preserving affine map mapping $[0,1]$ onto $I_{\pi(i)}^b$. If we furthermore define the vector $\rho_T:= (\rho_T^1, \dots , \rho_T^d)\in \mathbb{R}^d$ so that each entry  $\rho^i_T $ is given by 
$\rho^i_T:={ |I_i^b|}/{|I_i^t|} = {|T(I_i^t)|}/{|I_i^t|}$, 
the GIET $T$ is uniquely defined by the data $(\pi,\lambda_T, \rho_T ,\varphi_T )$ (there are essentially the \emph{shape-profile} coordinates\footnote{The \emph{shape} of a GIET is the unique \emph{affine} IET with combinatorics $\pi$ whose derivative $DT$ restricted to  
$I^t_i$ is a constant $DT_i$ such that $\log DT_i= {\rho_T^j}$. The shape is determined by the combinatorial datum $\pi$, the lenght vector $|\lambda_T$  
  and the vector $\rho_T$, known as \emph{log-slope} vector, see \cite{SU}.} used in \cite{Selim:loc, SU}). 

\smallskip
\noindent {\it The $\mathcal{C}^2$ distance.} 
To define the $\mathcal{C}^2$ distance on  $\mathcal{X}_d^r$ with $r\geq 2$, we will use this identification and the   $\mathcal{C}^2$ on each profile coordinate. 
Given  $f : [0,1] \longrightarrow \mathbb{R}$ of class $\mathcal{C}^2$, let  
$ ||f||_{\mathcal{C}^2} = \max_{0\leq i \leq 2}{||f^{(i)}||_\infty} $, where $f^{(i)}$ is the $i$-th derivative of $f$ and $|| \cdot ||_{\infty}$ denotes the sup norm. We extend this norm to $(\mathcal{C}^r([0,1], \mathbb{R}))^d$ simply by taking the sum of the norms on each coordinate, so $||(f_1,\dots, f_d)||_{\mathcal{C}^2}:= \max_{1\leq j\leq d}||f||_{\mathcal{C}^2} $

Given $T_1, T_2\in \mathcal{X}_d^r$ we set $d_{\mathcal{C}^2}(T_1,T_2) = +\infty$ if $T_1,T_2$ have different combinatorial data, i.e.~$T_i\in \mathcal{X}_{\pi_i}^r$ with $\pi_1\neq \pi_2$. If $T_1,T_2\in  \mathcal{X}_{\pi}^r$ and their coordinates are  $(\pi,\lambda_{T_1}, \rho_{T_1} ,\varphi_{T_1} )$ and  $(\pi,\lambda_{T_2}, \rho_{T_2} ,\varphi_{T_2} )$ respectively, we set
$$
d_{\mathcal{C}^2}(T_1, T_2):= \max_{1\leq j\leq d} |\lambda_{T_1}^j-\lambda_{T_2}^j|  + \max_{1\leq j\leq d} |\rho_{T_1}^j-\rho_{T_2}^j| + 
||\varphi_{T_1}-\varphi_{T_2}||_{\mathcal{C}^2}.
$$  
Notice that given a sequence $(T_n)_{n\in \mathbb{N}} $ where  $T_n$ has coordinates $(\pi_n,\lambda_{T_n}, \rho_{T_n} ,\varphi_{T_n} )$, if   $d_{\mathcal{C}^2}(T_n, T)\to 0$ as $n$ grows where $T$ has coordinates $(\pi,\lambda_{T}, \rho_{T} ,\varphi_{T} )$ when eventually $\pi_n=\pi$ and  $\rho_{T_n}\to \rho_{T}$ in $\mathbb{R}^d$, so it follows that $\lambda_{T_n}\to \lambda_{T}$ in $[0,1]^d$ and, for each coordinate of the profile,  $\varphi^i_{T_n} \to \varphi^i_T$ in $Dif\! f^2([0,1])^d$ for each $1\leq i\leq d$.

\smallskip
\noindent {\it Boundary of a GIET}. We conclude this background subsection recalling briefly a geometric definition of \emph{boundary} $B(T)$ of a GIET $T$ (first defined combinatorially and used in the work of Marmi, Moussa and Yoccoz \cite{MMY}). While we chose to include this definition for completeness since $B(T)$ appears in the statements of the main theorems, the boundary will not be used in the rest of the paper so the reader who desires to do so  can skip the rest of this subsection and move to \S~\ref{sec:renormalisation}.

Let $T  \in \mathcal{X}^r_d$ be a GIET. Let $\mathcal{U}:= \{ u_i^t, 0\leq 1\leq d\}$ be the top singularities of $T$. These can be \emph{partitioned} into $\kappa$ subsets defined as level sets of a map $s: \mathcal{U}\to \{1,\dots, \kappa\}$, where, if $T$ is the Poincar{\'e} section of a foliation $\mathcal{F}$ on a surface of genus $g$, $\kappa$ is the number of \emph{singularities} of the foliation\footnote{A generalized interval exchange map $T$ can be \emph{suspended} \cite{Ma:int, Yoc:Clay}  to an orientable (singular) \textit{foliation} $\mathcal{F} = \mathcal{F}(T)$ on a closed oriented surface $S = S(T)$ such that the singular points of $\mathcal{F}$ are {(possibly degenerate) saddles (with an even number of prongs)}; $T$ can then be recovered from $\mathcal{F}$ by considering a first-return map on a suitably chosen transverse arc $J$. Genus and number of saddles and prongs are fully determined by the combinatorial datum $\pi$ (see e.g.~\cite{Yoc:Clay} or \cite{Vi:IET}). 
Choosing  $J$ with endpoints at singularities, or on singular leaves, guarantees that the number of exchanged intervals of $T $ is as small as possible. In this case if $g$ is the genus of $S$ and $\kappa$ the cardinality of singularities, we have that  $ d = 2g + \kappa - 1.$} Geometrically, 
if we label the singularities of $\mathcal{F}$  by $\{1,\dots, \kappa\}$, the value $s(u_i)$ is the label of the singularity which corresponds to $u_i$, namely the singularity which is hit\footnote{Let $\mathcal{F} $ be a foliation on $S$ which suspends $T$, chosen so that both endpoints
  suspension  of $I$ are singular points of $\mathcal{F}$. In this case, one can show that all the singularities of $T$ are obtained by pulling-back a singular leaf of $\mathcal{F}$, i.e.~they are obtained as first return of a \emph{backward} leaf emanating from a singularity to the transveral.} 
	by the leaf of $\mathcal{F}$ emanating from $u_i$. 
A fully combinatorial definition of this map is given in \cite{MarmiYoccoz} (see also \cite{Yoc:Clay}). 

 Given $T\in \mathcal{X}_d^r$ with $r\geq 1$,  to define its boundary we consider its derivative $DT$ and set $f:= \log \D T$. By definition of a GIET, both $DT$ and $f$ are  piecewise continuous functions, continuous on each $I^t_i$ and (since $T$ extends to a differentiable diffeo on each $I^t_i$) the right and left limits of each of their branches exist. We denote    
 $f^r(u_i)$ and $f^l(u_i)$ respectively the right and left limits of $f$ at the discontinuity point $u_i$ for $1\leq i\leq d-1$.  We also set by convention $f^l(u_0):=0$ and $f^r(u_d):=0$. 


We define now the \emph{boundary} $B(T) $ of $T$ {(see also \cite{MMY2}), which is a vector  $B(T)=(B_s(T))_{1\leq s\leq \kappa}\in \mathcal{R}^\kappa$ given by 
$$ B_s(f) := \sum_{0\leq i\leq  d \ \text{s.t}\ s(u_i) = s}\left( {f^r(u_i) }-
{f^l(u_i)}\right).$$
In the notation by Marmi-Moussa-Yoccoz (see e.g.~\cite{MMY3}), $B(T):=\partial f= \partial \log \D T $
 where $\partial : \mathcal{C}_0\big(\sqcup_i{I_i^t(T)} \big)\to \mathbb{R}^\kappa$ is the \emph{boundary operator} introduced by Marmi-Moussa-Yoccoz in \cite{MMY3} (see also \cite{Yoc:Clay}).

\subsection{Renormalisation of GIETs}\label{sec:renormalisation} 
We will consider a renormalisation operator $\mathcal{R}$ on the space  $\mathcal{X}^r_d$ of GIET  defined  on the subspace of $d$-GIETs, $d\geq 2$, with no connections. Given a Keane GIET $T$, for every $n\in\mathbb{N}$, $\mathcal{R}^n T$ is another Keane GIET in $\mathcal{X}^r_d$  obtained  by rescaling the first return map of $T$ on an interval $I_n$ of the form $[0,\lambda_n]$. The operator $\mathcal{R}$ is therefore defined if we assign an algorithm which given a Keane $T$ allows to construct a sequence of nested intervals  $\{ I_n, \ n\in\mathbb{N}\} $  of the form $I_n= [0,\lambda_n]$ (so they all share zero as a common left endpoint) so that, if we denote $T_n$ the first return map of $T$ on $I_n$, $T$ is a  $d-$GIET. Let us denote by $(I^j_n)_{1\leq j\leq d}$ the subintervals exchanged by $T_n$. Then $\mathcal{R}^n T$ is the GIET in  $\mathcal{X}^r $  obtained rescaling linearly $T_n$, 
explicitely given by 
\begin{equation}\label{eq:renormalizedmap}
 \mathcal{R}^n T (x):= 
{T_n(\lambda_n x)}/{\lambda_n},
  \qquad \mathrm{for\ all} \ x\in 
\end{equation}

\smallskip
\noindent {\it Rauzy-Veech induction and its accelerations.}  The \emph{slowest}\footnote{Rauzy-Veech induction produces, for every Keane $T$, the sequence $(\lambda_n)_{n\in\mathbb{N}}\subset [0,1]$ which is \emph{maximal} with respect to inclusion, i.e.~it produces \emph{all} $\lambda\in [0,1]$ (in decreasing order) such that the induced map of $T$ on the interval $[0,\lambda]$ is again a $d$-IET (one can show indeed that these are countable). This key property of Rauzy-Veech induction was shown by Rauzy in his seminal work \cite{Ra:ech}.} algorithm with these properties is known as \emph{Rauzy-Veech induction} and will be denoted by $\mathcal{V}$. Since  the explicit definition of Rauzy-Veech induction will not play any role in what follows, we do not recall it here (and refer the interested reader for example to the lecture notes \cite{Yoc:Clay} or \cite{Vi:IET}). When a Keane $T$ is linearizable, one can show that $\lim_{n\to \infty}\lambda_n = 0$ (notice though that this is not true in general for GIETs\footnote{One can show more generally that $\lim_{n\to \infty}\lambda_n = 0$ exacly when the \emph{rotation number} of $T$ in the sense of GIETs, namely the path on the Rauzy diagram which has as vertices the combinatorial data $(\pi_n)_n$ of $(\mathcal{R}^n T)_n$ is \emph{infinitely-complete}, namely all labels $\{ 1, \dots, d\}$ of the subintervals $(I^t_j)_j$ and  $(I^b_j)_j$ appear as \emph{winners} of an elementary step of Rauzy-Veech induction infinitely many times (see the lecture notes \cite{Yoc:Clay} by Yoccoz for definitions and details). If $T$ is Keane but not semi-conjugated to a standard IET, it can be infinitely renormalizable, but its rotation number may fail to be infinitely complete and in this case it is possible that $\lim_{n\to \infty}\lambda_n = \lambda_\infty>0$.}). 
More renormalisation operators 
are obtained \emph{accelerating} $\mathcal{V}$ i.e.~their iterates have the form $\mathcal{R}^k(T):=\mathcal{V}^{n_k}(T)$ where $(n_k)_k$ is a suitably chosen sequence of iterates of $\mathcal{V}$ (which depends on $T$). 
 Classical accelerations of $\mathcal{V}$ are the \emph{Zorich acceleration} (which we will denote by $\mathcal{Z}$, which is important since $\mathcal{Z}: \mathcal{I}_d\to \mathcal{I}_d$ admits an absolutely continuous \emph{finite} invariant measure) and the \emph{positive} acceleration defined in \cite{MMY3} (the slowest which gives strictly \emph{positive} incidence matrices, see \S~\ref{sec:DC}). It is the latter which we will use as renormalisation operator in this paper.  

\smallskip \noindent {\it Dynamical partitions and Rohlin towers.} 
Let $T$ be a GIET such that the orbit $(\mathcal{R}^n(T))_{n\in\mathbb{N}}$ is well defined. 
The chosen renormalisation algorithm operator allows to produce a sequence of  \textit{dynamical partitions} and Rohlin towers presentations, defined as follows. 
Let $(I_n)_n$ be the nested sequence of inducing intervals and let  $I_n^j$, for $j=1,\dots, d$  be the subintervals exchanged by the first return map $T_n$ of $T$ on $I_n=[0,\lambda_n]$. For each 
$1\leq j\leq d$, $T_n$ restricted to $ I_n^j$ is equal to $T^{q_n^j}$  where $q_n^j$ is the first return time of $I_n^j$ to $I_n$ under $T$, i.e.~the minimum $q\geq 1$ such that $T^{q}(x)\in I_n$ for some (hence all) $x\in I_n^j$.
 Let us define  the {dynamical partition} $\mathcal{P}_n$ of \emph{of level} $n$ by
$$ \mathcal{P}_n := \bigcup_{j=1}^d{\mathcal{P}^j_n}, \qquad \text{where}\quad \mathcal{P}^j_n := \{ I_n^j, T(I_n^j),  T^2(I_n^j), \cdots, T^{q_{n}^j-1}(I_n^j)  \}. $$
 One can verify that $\mathcal{P}_n$ is a partition of $[0,1]$ into subintervals  and that, for each $1\leq j\leq d$, the collection  $\mathcal{P}^j_n$ is a \emph{Rohlin tower} by intervals, i.e.~a collection of disjoint intervals which are mapped one into the next by the action of $T$.  
 We say that the number $q_n^j$ of intervals in a tower is the \emph{height} of the (Rohlin) tower $\mathcal{P}^j_n$.  
Thus, $\mathcal{P}_n$ also gives a representation of $[0,1]$ as a \emph{skyscraper}, i.e.~a collection of Rohlin towers, for $T$.
Notice that if $n>m$, then the partition  $\mathcal{P}_n$  is a refinement of $\mathcal{P}_m$. 

\smallskip
Let us denote by $\mesh ({P}_n) $ the \emph{mesh} of the partition $\mathcal{P}_n$, namely
$\mesh({\mathcal{P}_n}): =\sup \{ |I|,$ where $ I$  is a floor of $\mathcal{P}_n\}$.
We record for future use the following observation.
\begin{remark}\label{meshtozero}
If a $T$  is topologically conjugate to a minimal Keane IET $T_0$, then $\mesh ({{P}_n}) $ goes to zero as $n$ grows. This follows simply because minimality is preserved by topological conjugacies and if \emph{mesh} $\mesh{{P}_n} $ fails to go to zero, the  dynamical towers for $T$ (which are well defined for a Keane GIET) would yield a \emph{wandering interval}, i.e.~one could find $J$ with $|J|>0$ such that the iteratates $T^n(J)$ with $n\geq n_0$ are all disjoint, contradicting minimality.
\end{remark}

\smallskip \noindent {\it Incidence matrices.} 
Given a renormalisation operator, such as $\mathcal{V}$, $\mathcal{Z}$ or any $\mathcal{R}$ defined as one of its acceleartions, we can define a sequence of \emph{incidence matrices} $(A_n)_{n\in\mathbb{Z}}$, where each $A_n \in SL(d,\mathbb{Z})$ is a $d\times d$ matrix with integer, non-negative entries, as follows.
If  $(T_n)_n$ is the sequence of IETs obtained inducing on the sequence $(I_n)_{n}$ of intervals given by the induction, 
then  the entry $(A_n)_{ij}$ of the $n^{th}$ incidence matrix $A_n$ gives the number of visits of  the orbit of $x\in I_n^j$ under $T_n$ to $I_n^i$ up to its first return to $I_{n+1}$. 

The incidence matrix entries have also an interpretation in terms of  of Rohlin towers: the Rohlin towers at step $n+1$ can be obtained by a \emph{cutting and stacking}\footnote{We do not give here a precise definition of \emph{cutting and stacking}, which is a standard construction in the study of ergodic theory and in particular of \emph{rank one} and, more in general, \emph{finite rank} dynamical systems.}, which construction from the Rohlin towers at step $n$: more precisely, for any $n\in\mathbb{N}$ and $1\leq i,j\leq d$, the Rohlin tower over $I^j_n$ is obtained stacking \emph{subtowers} of the Rohlin towers over $I^j_n$ (namely sets of the form $\{ T^k J, 0\leq k<q^{(n)}_j\}$ for some subinterval $J\subset I^j_n$). 
 Then  $ (A_n)_{ij}$ is the number of subtowers of the Rohlin tower over $I^i_n$ inside the Rohlin tower over $I^j_{n+1}$. It follows that the Rohlin tower over  $I_{n+1}^j$ is made by stacking exactly  {\color{black}$\sum_{i=1}^d (A_n)_{ij}$} subtowers of Rohlin towers of step $n$. Notice that {\color{black}$\sum_{i=1}^d (A_n)_{ij}$} is the sum of the entries
of the $j^{th}$ \emph{column} of the matrix $A_n$. 

In view of this remark, one can see that the vector column vectors $(q_n)_{n\in \mathbb{N}}$ which has as entries the heights $q_n^j$ (i.e.~return times), for $1\leq j\leq d$,  of Rohlin towers for $T$ with base $I_n$ satisfy the relations
\begin{equation} \label{heightsrelation}
{q}_n =  A(m,n) \,  q_m, \qquad \mathrm{for}\ 0\leq m\leq n, \ \  \textrm{where} \ \ A(m,n):= A_{n-1} A_{n-1} \cdots A_m
\end{equation}
and $q_0$ is by convention the vector with all entries $q^j_0=1$.  
On the other hand, the length (column) vectors $(\lambda_{n})_{n\in\mathbb{N}}$ that give the lengths $\lambda_{n}^j=|I_n^j|$ of the exchanged intervals of the induced map $T_n$ on the sequence of inducing intervals $\{ I_n, n\in\mathbb{N}\}$ given by the Zorich acceleration 
are governed by the \emph{transpose} matrices $(A_n^\dag)_n$  (where $A^\dag$ denotes the transpose of $A$) of the incidence matrices $(A_n)_n$, namely we have
\begin{equation}\label{lengthsrelation}
\lambda_m= A(m,n)^\dag \lambda_n, \qquad \text{where}\quad A(m,n)^\dag:= A_{m}^\dag A_{m-1}^\dag \cdots A_{n-1}^\dag , \qquad \text{for\ every}\  0\leq m<n.
\end{equation}

\smallskip
 Notice that using of the norm $\Vert v \Vert= \sum_{j}|v_j|$ on a column  vector $v$ with entries $v_j$, $1\leq j\leq d$, and the norm $\norm{A} = \sum_{i,j}|A_{ij}| $ on matrices, since return times as well as lenghts are positive numbers and $\Vert A^\dag\Vert = \Vert A\Vert$, we get the two following relations 
\begin{equation}\label{normproduct}
\max_j \, q_{n}^j \leq \Vert q_{n} \Vert  \leq \Vert A{(m,n)} \Vert q_{m} \Vert, \qquad \text{for\ any}\ 0\leq  m<n, 
\end{equation}
\begin{equation}\label{lenghts_uppbound}
\max_j \, \lambda^{m}_j \leq \Vert \lambda_{m} \Vert  \leq \Vert A{(m,n)^\dag } \Vert \, \Vert \lambda_{n} \Vert = \Vert A{(m,n)} \Vert \, \Vert \lambda_{n} \Vert, \qquad \text{for\ any}\ 0 \leq m<n. 
\end{equation}
Moreover, if the matrix $A_{n+1}>0$ is \emph{positive}, i.e.~all its entries $(A_{n+1})_{ij}$ are strictly positive, one can also show (see e.g.~\cite{MMY} or \cite{Yoc:Clay}) that $\min_j \, \lambda^{n+1}_j \geq {\max_i \, \lambda^{n}_i }/{\Vert A_{n+1}\Vert }$ and therefore, since ${\max_i \, \lambda^{n}_i }\geq \Vert \lambda^{n}\Vert/d$, from \eqref{lenghts_uppbound} we get also
\be\label{lenghts_lowbound}
\min_j \, \lambda^{n+1}_j \geq  \frac{\Vert \lambda^{n}\Vert}{d \Vert A_n\Vert}
\geq 
  \frac{\Vert \lambda^{m}\Vert}{d \Vert A_n\Vert   \Vert A{(m,n)} \Vert }, 
 \qquad \text{for\ any}\ 0\leq m<n. 
\ee

\smallskip
When the renormalisation operator is Rauzy-Veech induction $\mathcal{Z}$, we will use the notation $(Z_n)_n$ for the incidence matrices of Rauzy-Veech induction and $Q(n,m):= Z_{n-1}\cdots Z_m$ for their products. Notice that following \emph{cocycle relation} then holds for any triple of integers $n,m,p$:
\begin{equation}\label{cocyclerel}
Q(n,p)=Q(m,p)\, Q(n,m), \qquad \text{for\ all}\ n<m<p.
\end{equation}    
\noindent  The map  $Z$: $\mathcal{I}_d \rightarrow SL(d,\mathbb{Z})$ is  indeed a  \emph{cocycle} over $\mathcal{Z}$, that we call the \emph{Zorich cocycle}\footnote{Notice that there is also another cocycle, also sometimes called Zorich cocycle, which transforms \emph{lengths} and is actually the transpose inverse of the cocycle here defined.} (also sometimes referred to as \emph{Kontsevich-Zorich} cocycle).

\smallskip
\noindent 
{\it Birkhoff sums and special Birkhoff sums.}\label{sec:sums}
Let $T$ be a $d$-GIET with no connections and let $f:[0,1]\to \mathbb{R}$ be an observable. We will assume that $f$ is piecewise-$\mathcal{C}^{1}$, more precisely that its restriction $f_i$ to each  continuity interval $I^t_i$ for $T$, for each $1\leq i\leq d$, is continuous   and extends to a differentiable function on the  closure $\overline{I^t_i}$. In this paper we will be interested in studying the Birkhoff sums of the function $f:=\log DT$; notice that  when $T\in \mathcal{X}^r_d $ with $r\geq 1$,  $f=\log DT$ satisfies the above properties.

\smallskip
\noindent For each $n\in \mathbb{Z}$, we define the $n^{th} $ \emph{Birkhoff sum} of $f$ over $T$ as the functions
\begin{equation}\label{def:BS}
S_n f:= 
\begin{cases} \sum_{j=0}^{n-1}f \circ T^j , & \text{if}\  n>0;\\
0 , & \text{if}\  n=0;\\
\sum_{j=1}^{n}f \circ T^{-j} , & \text{if}\ n<0;\\
\end{cases}
\end{equation}
The definition of the Birkohff sums $S_nf$ for $n\leq 0$ is given so that $(S_nf)_{n\in\mathbb{Z}}$ are a $\mathbb{Z}$-additive cocycle, i.e.~satisfy 
$$S_{n+m}\, f = S_n\,  f +  S_m f\circ T^n,\qquad \text{ for\ all \ } n,m \in \mathbb{Z}.$$
Notice also that, for any  $m \in \mathbb{N}$ 
\begin{equation}\label{backwardsBS}
S_m f(x ) = S_{-m} f (T^{m} u), \qquad \text{for \ all}\ x \in [0,1].
\end{equation}

\smallskip
The Birkhoff sums $S_n f$, $n\geq 0$, can be studied via renormalisation exploiting the notion of \emph{special Birkhoff sums} that we now recall. 
If $\{I_n, \ \ n\in\mathbb{N}\} $ 
 is the sequence of inducing intervals given by the renormalisation algorithm, 
the (sequence of) \emph{special Birkhoff sums} $f_n$, $n\in\mathbb{N}$, is  the sequence of functions 
$f_n: I_n \to \mathbb{R}$ 
obtained \emph{inducing} $f$ over the first return map $T_n$, namely given by
\begin{equation} \label{eq:sBS}
f_n(x):= S_{q^j_n}(x) =\sum_{\ell=0}^{q_n^j-1} f \left(T_n^\ell( x)\right), \qquad \text{if}\ x\in I^j_{n},\quad \text{for\ any}\ 1\leq j\leq d, \ n\in\mathbb{N},
\end{equation}
where $q_n^j$ is as above the first return time of any $x\in I_n^j$ to $I_{n+1}$. Thus one can think of $f_n(x)$ as the Birkhoff sum of $f$ \emph{along the Rohlin tower} of height $q_n^j$ over $I_n^j$. 

Given a the $n^{th}$-special Birkhoff sum $f_n$,  we can build $f_{n+1}$ from $f_n$ and $T_n$ by writing 
\begin{equation}\label{SBSdecomp}
f_{n+1} (x) = \sum_{k=0}^{(A_n)_{ij}-1} f_{n} \left(({T}_n)^k(x)\right), \qquad  \text{for\ any}\ x\in I_{n+1}^j,
\end{equation}
where $(A_n)_{ij}$  is the $(i,j)$ entry of the $n^{th}$ incidence  matrix $A_n$. 
  This relation, which can be proved simply recalling the definition of special Birkhoff sums and Birkhoff sums, can be understood  in terms of \emph{cutting and stacking} of Rohlin towers: the relation indeed mimics at the level of special Birkhoff sums the fact (recalled previously) that the Rohlin tower over  $I^j_{n+1}$ is obtained by stacking $(Z_n)_{ij}$ subtowers of the Rohlin tower over  $I^{n}$ and hence, correspondingly, the special Birkhoff sum $f_{n+1} (x) $ is obtained as sum of $(Z_n)_{ij}$ values of the special Birkhoff sum $f^{n}$ at points of $I^i_{n}$. 

\smallskip 
\noindent {\it Decomposition of Birkhoff sums.} 
Special Birkhoff sums can be used as follows as fudamental \emph{building blocks} to study Birkhoff sums, see e.g.~\cite{Zo:dev, MMY, Ul:mix, Ul:abs, MarmiYoccoz, MUY}.  In  the special case in which $x_0\in I_{n_0+1}$ 
 for some $n_0\in\mathbb{N}$, 
 if follows (exploiting recursively \eqref{SBSdecomp}) that, if $1\leq j\leq d$ is such that $x_0\in I^{(n_0+1)}_j$, the Birkhoff Birkhoff sum $S_r(x_0)$  for any $0\leq r\leq q^{(n_0+1)}_j$ can be decomposed into special Birkhoff sum 
as 
\begin{equation}\label{geometric}
S_r f(x_0)=\sum_{n=0}^{n_0} \sum_{\ell=0}^{a_{n}-1}  f_n \left(x_n^\ell \right) , \qquad \text{where} \ 0\leq a_n\leq \Vert A_n\Vert , \quad x_n^\ell \in I_n,\ \text{for}\ 0\leq \ell \leq a_n -1.
\end{equation}
We will refer to \eqref{geometric} as \emph{geometric decomposition} of $S_r f(x_0)$. 
For the general case of a  Birkhoff sum $S_r f(x)$ for any $x\in [0,1]$ and $r\in\mathbb{N}$, we can define   
$n_0=n_0(x,r)$ to be the maximum $n_0\geq 1$ such that $I^{(n)}$ contains at least \emph{two} points of the orbit $\{T^i x, 0\leq i<r\}$. (This guarantees that $r$ is larger than the smallest height of a tower over $I^{(n_0)}$, but at the same time that it is smaller than a over $I^{(n_0+1)}$. Then, if $x_0 =T^{i_0}(x)$ is one of the points in $I^{(n_0)}$ we can split the Birkhoff sum $S_r f(x)$ into two sums of the previous form, one for $T$ and the other for $T^{-1}$. From this geometric decomposition, we then get the following estimate: 
\begin{equation}\label{geometricestimate}
|S_r f(x)|\leq 2 \sum_{n=0}^{n_0}  \Vert A_n\Vert \, \Vert f_n\Vert, \qquad \text{for any}\ x\in [0,1]\, \backslash \, I_{n_0+1}.
\end{equation}  
}

\subsection{Non-linear tools}\label{sec:nl}
For any $\mathcal{C}^2$ map $f: I \longrightarrow J$ where $I$ and $J$ are open intervals such that $\mathrm{D}f$ does not vanish, one can define the \emph{non-linearity}\footnote{The function $\eta_f$ is called non-linearity since it measures how far $f$ is from being affine and has the property that $\eta_f \equiv 0 $ if and only if $f$ is an affine map.} function $\eta_f$ to be the function $\eta_f:I\to \mathbb{R}$ given by 
\begin{equation}\label{def:nonlinearity}
\eta_f(x):= (\mathrm{D} \log \mathrm{D}f)(x) = \frac{\mathrm{D}^2f(x)}{\mathrm{D}f(x)}.
\end{equation} 

The non-linearity satisfy the following \emph{distribution property}: if $f: I \longrightarrow J$ and $g : J \longrightarrow K$ are diffeomorphisms of class $\mathcal{C}^2$, then
\begin{equation}\label{NLdistribution}
 \int_{I}{\eta_{g \circ f}} = \int_{I}{\eta_f}  +  \int_{J}{\eta_g}.
\end{equation}

 Given a $\mathcal{C}^2$ interval exchange map  $T: [0,1]\to [0,1]$, defined on the continuity intervals $I_j\subset [0,1]$  we  define the non-linearity $\eta_T$ to be the (bounded) piecewise continuous map from $[0,1]$ to $\mathbb{R}$ given by 
$$\eta_T(x):=\eta_{T_j}(x), \qquad \mathrm{if} \ x\in I_j, \quad 1\leq j\leq d,$$ 
where $T_j: I_j \to [0,1]$ are the  branches of $T$ obtained restricting $T$ to its continuity intervals. 

We subsequently define the \emph{total non-linearity} $|N|(T)$ of $T$ to be
$$ 
|N|(T) := \int_0^1{|\eta_T(x)|dx}.$$ 
\noindent For any $r\geq 2$, $\pi \in \mathfrak{S}_r$, 
the  total non-linearity 
is invariant under rescaling by (restrictions) of affine maps, so that in particular if $a,b$ are  (restrictions of) linear maps, 
$\overline{N}(a\circ T\circ b) = \overline{N}(T)$.


\section{Regularity of the conjugacy }\label{sec:proof}
In this section we prove $\mathcal{C}^{1+\alpha }$-regularity of the conjugacy to a full measure set of IETs under the assumption that there is exponential convergence of renormalisation, i.e.~we prove Theorem~\ref{thmb}. We begin the section explain in \S~\ref{sec:strategy} the outline of the proof. 

\subsection{Outline of the proof} \label{sec:strategy}
Let us  consider a Keane standard IET $T_0$ and  assuming that the GIET $T$ and $T_0$ are differentiably conjugate. Let $h$ be the diffeomorphism of $[0,1]$ such that $h\circ T=T_0\circ h$. 
To show that, for some $0<\alpha<1$, $h$ 
is $ \mathcal{C}^{1+\alpha}$, we then have to show that it derivative $Dh$ is $\alpha-$H{\"older} continuous.  
We will first consider the function $\varphi:=\log Dh $ and show that $\varphi$ is $\alpha$-H{\"older} for some $0<\alpha<1$, namely
that there exists $C>0$ such that it satisfies 
$$|\log Dh(x)- \log Dh(y)|=|  \varphi(x)- \varphi(y)|\leq C |x-y|^\alpha
$$
for every $x,y\in I_0$.  Thus, since the RHS as well as $h$ are bounded, the same estimate holds (up to changing the constant) also for $h$ and shows that also $h$  is $\alpha-$H{\"older} continuous.

In order to  estimate $|\varphi(x)-\varphi(y)|$  and 
compare it with $|x-y|$, we exploit renormalisation and more precisely a suitable acceleration of Rauzy-Veech induction (described below).  As discussed in \S~\ref{sec:diff}, instead than 
using  \emph{spatial decompositions} of the interval $[x,y]$ as in Marmi-Yoccoz work \cite{MarmiYoccoz}, 
we use \emph{time-decomposition} of Birkhoff sums of the function $f=\log DT$  (as in the classical case of circle diffeomorphisms, see e.g.~\cite{KT:Her}). Since we do not have the freedom of choosing $x_0$ (due to the lack of homogeneity of GIETs), we build \emph{approximations} to both $x$ and $y$ through points in $\mathcal{O}_T(0)$ (which we call \emph{single orbit approximations}, see \S~\ref{sec:appr} and in particular Propostion~\ref{prop:appr} for details). Even though in the (G)IET setting  the structure of dynamical partitions and the relation between space decompositions and first return times is not well understood as in the case of rotations, 
 through the use of the (two steps of) the positive acceleration, and the Rohlin-tower structure of IETs when we assume a Roth-type growth (which guarantees a good control on \emph{ratios} of towers lenghts and heights), we can construct suitable approximations by points $x_n\to x$ and $y_n\to y$ in $\mathcal{O}_T(0)$ for which we can replicate a control of Birkhoff sums as good as that available for circle diffeos (see in particular Propostion~\ref{prop:specialvariation}).

The acceleration of Rauzy-Veech induction that we exploit as renormalisation operator is the \emph{positive} acceleration: 
 when $T$ is Keane and linearizable,  there exists indeed a sequence $(n_k)_k$ of Zorich induction times such that the incidence matrices $A_{k}:=Q(n_{k-1}, n_{k})$ can be written as product of two strictly positive matrices. The Diophantine-like condition that we assume is that this sequence has subexponential growth, namely for any $\epsilon>0$ there exists $C_\epsilon>0$ such that
 $\Vert A_{k}\Vert \leq C_\epsilon e^{k\epsilon}$ for all $k\in \mathbb{N}$ and that their products have exponentialy growth (see Definition~\ref{def:DC} in \S~\ref{sec:DC} for details).  This type of growth is known to be typical, i.e.~this condition is of full measure (see Lemma~\ref{lemma:fullDC} in \ref{sec:DC}). 

Given any two $x,y\in  I_0:=[0,1]$, we will first choose a \emph{scale} of renormalisation, namely one of the accelerating times $n_{k_0}$ so that $|x-y|$ is comparable with the size of $I_{n_{k_0}}$ of the $n_{k_0}^{th}$ inducing subinterval (see \S~\ref{sec:spacescale}).  Exploiting the Diophantine-type assumption, one can show that the lenghts $|I_{n_{k}}|$ decay exponentially in $k$;  this will provide a lower bound of the form
\begin{equation}\label{expboundspace}
|x-y|\geq c \lambda_1^{k_0}, \qquad \text{where}\ 0<\lambda_1<1 ,
\end{equation}
see Lemma~\ref{lemma:denom} in \S~\ref{sec:spacescale}. 
We will then show (see Proposition~\ref{prop:numerator}) that 
\begin{equation}\label{expboundphi}
|\varphi(x)-\varphi(y)|\leq C \lambda_2^{k_0},\qquad  \text{for} \ 
 \varphi:= \log Dh,
\end{equation}
\noindent for some $0<\lambda_2<1$. 
Thus, if we define $\alpha:= \log \lambda_2/\log \lambda_1>0$, we have that
 $\lambda_1=\lambda_2^\alpha$. 
Then, combining \eqref{expboundphi} and \eqref{expboundspace}, we have that
$$
{|\varphi(x)-\varphi(y)|}\leq C {\lambda_2^{k_0}} =  C (\lambda_1^{k_0})^\alpha \leq C' |x-y|^\alpha
$$
for some $C'>0$. If $\alpha>1$, this forces $\varphi=\log DT$ to be constant, which is a contradiction, so we deduce now that $\alpha\leq 1$. Therefore we conclude that $\varphi:=\log Dh$ (and hence, as explained earlier, also $Dh$) is $\alpha$-H{\"older} with $0<\alpha \leq 1$. The heart of the proof is now to show that \eqref{expboundphi} holds.  

To estimate the difference $|\varphi(x)-\varphi(y)|$ we consider first differences of the form  $|\varphi(x^i)-\varphi(x^j)|$ where $x^i, x^j$ belong to the orbit of $0$ under $T$, namely $x^i=T^i(0)$ and  $x^j=T^j(0)$. Notice that
the function $\varphi:=\log Dh$ satisfies the cohomological equation 
\begin{equation}\label{eq:cohm}
\varphi(Tx) - \varphi(x) = f(x) \qquad \text{where} \ f(x):= -\log DT (x)
\end{equation}
(which follows from the conjugacy equation, differentiating with the chain rule and then taking log). Therefore,  in this special case 
 we can use that, assuming without loss of generality that $j>i$, 
$$
|\varphi(x^j)-\varphi (x^i)|=|\varphi (T^j(0))-\varphi (T^i(0))|=|S_{j-i} f (x^j)|.
$$
 To estimate the Birkhoff sum $S_{j-i} f (x^j)$  of the function $f:= -\log DT $, we will exploit a geometric decompositions into special Birkhoff sums (in the spirit of the decomposition recalled in the background section \S~\ref{sec:renormalisation}), which will show that there exists  some $k_0=k_0(j-i)$ such that
$$
|S_{j-i} f (x^j)| \leq  2 \sum_{k\geq k_0(j-i)} || A_{k+1}|| \, || A_{k+2}|| \Vert f_k\Vert_{\infty} .
$$
From this expression, we see that the estimates of $S_{j-i} f (x^j)$ depend (exponentially) on the \emph{smallest} $k_0=k_0(j-i)$ such that the special Birkhoff sum $f_{n_k}$ appears in the decomposition: indeed, one can show that exponential convergence of renormalisation implies that special Birkhoff sums $f_k$ decay exponentially (see Lemma~\ref{SBSviaR}) and the Diophantine-like assumption that  $\{ \Vert A_k\Vert, \ \ k\in \mathbb{N}\}$ grow subexponentially, ensures that also the above series is geometric (and hence has order of the largest term, which corresponds to the smallest $k$ involved, namely $k_0=k_0(i-j)$). A hidden technical difficulty, in this part, is that the Birkhoff sum $S_{j-i} f (x^j)$ cannot be decomposed into full Birkhoff sums with $k\geq k_0$ when $x^j$ is not a point of one of the inducing intervals. In the decomposition we use, we also use \emph{broken special Birkhoff sums} (see \S~\ref{sec:sameorbit}, in particular Definition~\ref{def:brokenBS}). To control these \emph{broken sums}, we then have to use non-linear estimates (namely estimating integrals of non-linearity, see \S~\ref{sec:nl}) and exploit the full force of the exponential convergence of renormalisation assumption in the $\mathcal{C}^2$-norm (see the proof of Proposition~\ref{prop:reminderestimate} for details).

The final and key part of the proof is then to show that we can approximate any $x,y\in [0,1]$ such that the \emph{scale} of $|x-y| $ is $n_{k_0}$ with sequences of points $x_n\to x$ and $y_n\to y$ which all belong to the orbit $\mathcal{O}_T(0)$ of zero and such that, for all $n$, the $|\varphi(x_n)-\varphi (y_n)|$ are Birkhoff sums of $f$ which can all be decomposed into special Birkhoff $f_{k}$ with $k\geq k_0$. Thus, the estimates for general $x,y$ can be deduced by estimates for points in the orbit $\mathcal{O}_T(0)$ of zero. We refer the reader to the beginning of \S~\ref{sec:Holder} for a more detailed outline of the steps in this construction and the proof of the desired H{\"o}lder estimates. 

This concludes the overview of the proof. 
We conclude the section with an overview of how the proof is organized into subsections.

\medskip
\noindent {\it Organization  of the proof.} The proof of Theorem~\ref{thmb} is organized as follows. 
We first describe, in \S~\ref{sec:DC},  the required Diophantine-like condition in terms of growth of the matrices of the positive acceleration of Rauzy-Veech induction (see Definition~\ref{def:DC}) and show that it has full measure (see Lemma~\ref{lemma:fullDC}). 
 In \S~\ref{sec:consequences} we prove some preliminary consequences of   the assumption of exponential convergence of renormalisation: we show that $\mathcal{C}^1$-exponential convergence  implies exponential convergence of the special Birkhoff sums of $f:=\log DT$, while $\mathcal{C}^1$-exponential convergence allows to control non-linearity 
 more specifically, integrals of non-linearity within a tower, see Proposition~\ref{prop:reminderestimate}). 
 The H{\"older} estimates for $\varphi:= \log Dh$ are proved in \S~\ref{sec:Holder}, first in the special case when $x,y$ belong to $\mathcal{O}_T(0)$, then, through a crucial construction of approximations $x_n\to x$ and $y_n\to y$, for general points $x,y$. A local overview of the strategy of this part of the proof is given at the beginning of  \S~\ref{sec:Holder}.

\subsection{The Diophantine-like condition}\label{sec:DC}
Let us now define the full measure condition on the linearization $T_0$ of $T$ which we require to prove Theorem~\ref{thmb}.  This condition simply guarantees that the \emph{growth rate} of the incidence matrices of the renormalisation operator $\mathcal{R}$ given by the \emph{positive acceleration} of Rauzy-Veech induction is \emph{typical} (see the proof of Lemma~\ref{lemma:fullDC} for details). We use in this definition the terminology and the notations introduced in the background section  \S~\ref{sec:renormalisation}.

\begin{definition}[Typical Positive Growth]\label{def:DC}
Let us say that a Keane GIET $T$  has \emph{typical positive growth} (or for short, TPG) if there exists  an increasing sequence of accelerating times  $(n_k)_{k\in \mathbb{N}}$ of Rauzy-Veech induction such that, the incidence matrices $(A_k)_{k\in \mathbb{N}}$ of the associated accelerated renormalisation algorithm 
 have the following properties:
\begin{itemize}
\item[(P)] {\emph{Positivity}}: for every $k\in \mathbb{N}$, $A_k>0$ is matrix with strictly positive entries;
\item[(S)] {\emph{Subexponential growth of the matrices}}: for every $\epsilon>0$, there exists $C_\epsilon>0$ such that
$$
\Vert A_k\Vert \leq C_\epsilon \, e^{\epsilon k}, \qquad \text{for\ all \ } k\in \mathbb{N};
$$
\item[(E)]{\emph{Exponential growth of their products}}:  there exists $\rho>1$ and $K>0$ such that
$$
\Vert A(0, k)\Vert  =\Vert A_{k-1} A_{k-2} \cdots A_{0} \Vert = \Vert Q(0, n_k)\Vert 
\leq K \rho^ k, \qquad \text{for\ all \ } k\in \mathbb{N}.
$$
\end{itemize}
\end{definition}
The existence of a sequence of times $(n_k)_{k\in \mathbb{N}}$  which produce a \emph{positive} acceleration, namely whose incidence matrices satisfy the positivity in $(P)$,  is a well known property of all (standard) IETs which satisfy the Keane condition (first  proved by Marmi, Moussa and Yoccoz in \cite{MMY}, see also \cite{Yoc:Clay, Vi:IET}). Thus, if a Keane GIET $T$ is linearizable, i.e.~topologically conjugated to a standard (Keane) IET  $T_0$, the existence of a sequence for which  $(P)$ is automatically guaranteed. On the other hand,  the request that the matrices $(A_k)_k$ (and their products) satisfy the restriction on their growth rates imposed by $(S)$ and $(E)$, 
impose what we call a \emph{Diophantine-like condition} on the (G)IET (see e.g.~the \cite{Ul:ICM} for a survey of conditions of this type).     
The  type of \emph{growth} requested by the TPG of Definition~\ref{def:DC} is well known to be \emph{typical}, i.e.~satisfied by  a full measure set of linearizations $T_0$:

\begin{lemma}\label{lemma:fullDC}
Lebesgue almost every $T_0$ in $\mathcal{I}_d$ has TPG (Typical Positive Growth).  
\end{lemma}
\noindent The proof of this Lemma uses  rather classical tools in the study of IETs  and relies essentially on Oseledets theorem (for the Zorich acceleration of Rauzy-Veech induction). We include the proof, but assume some classical tools and results from the theory of IETs beyond those recalled in the background section \S~\ref{background}, since they are only locally used for this proof (the interested reader can find more details e.g.~in \cite{Ul:mix}, see also the lecture notes~\cite{Yoc:Clay}). 
\begin{proof}[Proof of Lemma~\ref{lemma:fullDC}] 
Let $\mathcal{Z}:\mathcal{I}_d' \to \mathcal{I}_d'$ be Zorich acceleration \cite{Zo:gau} of Rauzy-Veech induction (defined for a full measure subset  $\mathcal{I}_d'\subset \mathcal{I}_d$ of Keane IETs in $\mathcal{I}_d$, see \cite{Zo:gau}) and let $Z(T)$ be the matrix associated to one step of Zorich induction. 
The map  $Z: \mathcal{I}_d\to SL(d,\mathcal{Z})$ is then a cocycle over $\mathcal{Z}$ known as \emph{Zorich cocycle} \cite{Zo:gau}.  

Let  $A>0$ be any fixed positive matrix $A>0$ which can occur as product of matrices of Zorich induction, i.e.~such that $A= Z(T) \cdots Z(\mathcal{Z}^n T )$ for some (Keane) IET $T$  with combinatorics $\pi$ and $n\in \mathbb{N}$. 
Consider the following subsimplex of the standard simplex $\Delta_d$:
$$
\Delta_A := \left\{ A \lambda, \quad \lambda\in \Delta \} \subset \Delta_d := \{ \lambda=(\lambda_1,\dots, \lambda_d) : \quad \lambda_i\leq 0 \ \text{for all}\ 1\leq i\leq d \ \text{and} \ \sum_{i=1}^d\lambda_i=1\right\}.
$$ 
Consider the first return map $\mathcal{Z}_A$ of Zorich induction to the subset $\{\pi\} \times \Delta_A$ of $\mathcal{I}_d$, which is well defined on a full measure set by Poincar{\'e} recurrence since $\mathcal{Z}$ preserves a finite invariant measure $\mu_\mathcal{Z}$ (the \emph{Zorich measure} whose finiteness was proved in \cite{Zo:gau}). Notice first of all that if we consider the $n^{th}$-power $\mathcal{R}:= \mathcal{Z}_A^n$ (which can be  extended to an operator almost everywhere defined on $\mathcal{I}_d$ by defining $\mathcal{R}(T)=\mathcal{V}^{r(T)}(T)$ where $r(T)$ is the first entrance time of $T$ to $\{\pi\}\times \Delta_A$ if $T$ is not already in $\{\pi\}\times \Delta_A$), then $\mathcal{R}$ is a positive acceleration, i.e.~its incidence matrices satisfy $(P)$ (taking the $n^{th}$ power guarantees that the whole matrix $A>0$ appears as prefix of the incidence matrix, i.e.~for every $k$ we can write $A_k= A B_k$ for some $B_k\geq 0$, so that $A_k>0$). 
  $\mathcal{R}:=\mathcal{I}_d'\to \mathcal{I}_d'$ almost everywhere defined given by 
$\mathcal{R}(T)=\mathcal{V}^{r}$

Since Zorich cocycle is integrable, i.e.~$\int_{\mathcal{I}_d} \log \Vert Z(T)\Vert \mathrm{d}\mu_{\mathcal{Z}}<+\infty$, also the accelerated cocycle $T\to A(T)$ over the acceleration $\mathcal{R}$ is   integrable. Thus, we can apply Oseledets theorem to conclude that, for any $\rho>\lambda_1$ where  $\rho_1>0$ is the top Lypaunov exponent of the accelerated cocycle (which actually satisfy $\rho_1>1$, because of the form of the Rauzy-Veech matrices) the growth rate prescribed by condition $(E)$ holds for almost every $T \in \mathcal{I}_d$. Now, the growth rate prescribed by condition $(S)$ for the matrices $(A_k)_{k}$ follows simply by the Birkhoff ergodic theorem.  

\end{proof}

\subsection{Consequences of convergence of renormalisation.}\label{sec:consequences}
Let $\mathcal{R}$ be any renormalisation operator in the sense of \S~\ref{sec:renormalisation}. Let $T$ be an \emph{infinitely renormalizable} GIET, i.e.~a GIET such that 
$\mathcal{R}^n(T)$ is well defined for every ${n\in \mathbb{N}}$ and let $(I_k)_k$ be the (infinite by assumption) sequence of inducing intervals produced by the renormalisation algorithm.  We will prove in this section two important consequences of the assumption that the orbit $(\mathcal{R}^n(T))_{n\in \mathbb{N}}$ of $T$ converges (exponentially fast) to the subspace of $\mathcal{I}_d$ of standard IETs, in $\mathcal{C}^1$ or $\mathcal{C}^2$ sense respectively (namely Lemma~\ref{SBSviaR} and Lemma~\ref{nLdecay} respectively). 

\smallskip
In the special case in which $f:=\log D T$, there is an important connection between  convergence of renormalisation in the $\mathcal{C}^1$-sense and convergence of the sequence $(f_k)_{k}$ of special Birkhoff sums inducing $f$ on $(I_k)_{k}$.

\begin{lemma}[$\mathcal{C}^1$ convergence and decay of special Birkhoff sums of $f=\log DT$, see \cite{SU}]\label{SBSviaR}{\color{black}
Let $T$ be an infinitely renormalizable GIET and let $f:=\log DT$. If 
$d_{\mathcal{C}^1} (\mathcal{R}^k(T), \mathcal{I}_d)$ converges to zero exponentially, then there exists $K>0, \rho<1$ such that 
$$\| f_k\|_{\infty}\leq K \rho^k, $$
i.e.~the sup-norm $\| f_k\|_{\infty}$ of the special Birkhoff sums $f_k$ on their domain $I_k$ also converges to zero exponentially.}
\end{lemma}
\noindent We refer to \cite{SU} for the proof of the Lemma.


\vspace{2mm}

\paragraph{\bf Non-linearity decay via renormalisation} 
Assuming that we have (exponential) convergence of renormalisation with respect to the $\mathcal{C}^2$ distance (defined in  \S~\ref{giet}) is necessary to control the contribution of non-linear terms, in particular to control the decay of the (total) non-linearity. 
For the regularity estimates in the following \S~\ref{sec:Holder}, we will in particular need the following result. 

\smallskip

\begin{lemma}[$\mathcal{C}^2$ convergence and decay of non-linearity]\label{nLdecay}
Assume that $T$ is such that
for some $K_1>0$ and {$0<\rho_2<1$},  we have that $ d_{\mathcal{C}^2}(\mathcal{R}^{n}(T), \mathcal{I}_d ) \leq K_1 \, \rho_2^n$ for every $k\in \mathbb{N}$. 
Then, we have that a constant $C'$ such that
$$
\Vert \eta_{\mathcal{R}^n T}\Vert_{\infty} \leq C' \rho_2^n.
$$
\end{lemma}

\begin{proof}

This is a direct consequence of the fact that $\eta_T := \frac{D^2T}{DT}$.

\end{proof}

\subsection{H{\"o}lder estimates}\label{sec:Holder}
In this section we now show that $\varphi:=\log DT$ is H{\"o}lder. Let us first give an overview the steps of the proof. 
The proof is split in three main steps (presented in separate subsections). 

\smallskip \noindent {\it Step 1 (\S~\ref{sec:spacescale}): choice of the space scale}. 
Given $x,y\in I_0:=[0,1]$, in \S~\ref{sec:spacescale} we first of all choose a renormalisation time $k_0$ (which we call renormalisation \emph{scale}) so that $|x-y|$ is comparable to the lenght $|I_{{k_0}}|$ of the $k_0^{th}$ inducing interval given by the chosen renormalisation algorithm (see \S~\ref{sec:DC}) and estimate the lenght $|x-y|$ using renormalisation (see Lemma~\ref{lemma:denom}). 

\smallskip \noindent {\it Step 2 (\S~\ref{sec:sameorbit}): $x,y$ in the same orbit}.  
In  \S~\ref{sec:sameorbit} we estimate $|\varphi(x)-\varphi(y)|$ in the special case in which $x,y$ both belong to the orbit $\mathcal{O}_T(0):= \{ T^n 0, n\in \mathbb{N}\}$ of $0$ under $T$.  If $x=T^i 0 $ and $y=T^j 0$, we  call \emph{time-distance} the difference  $|i-j|$ and we say that  \emph{order} of this difference is at least $n$ if $|i-j|\geq q_n$ where $q_n:=\max_{1\leq j\leq d} q_n^j$. 
 The estimate which we prove on $|\varphi(x)-\varphi(y)|$ depends on the \emph{order} $n$ of the \emph{time-distance} between $x,y$ and we show that they are exponentially small in $n$. 

\smallskip \noindent {\it Step 3 (\S~\ref{sec:appr}): General $x,y$}.   Finally, to reduce the general case where $x,y \in \I_0$ are arbitrary to the above special case in which $x,y\in \mathcal{O}_T(0)$, in \S~\ref{sec:appr} we  construct two sequences $(x_n)_n, (y_n)_n$ which approximate $x$ and $y$, in the sense that $\lim_{n\to \infty}x_n=x$ and $\lim_{n\to \infty}x_n=x$ and such that, for any $n\in\mathbb{N}$, $x_n,y_n\in \mathcal{O}_T(0)$ and have order at least $n_{k_0}$ (where $k_0$ is the scale of $|x-y|$, see above). Thus, using continuity and the previous step, we obtain the desired estimates for $|\varphi(x)-\varphi(y)|$ in the general case. 

\subsubsection{Space scale choice and space estimates}\label{sec:spacescale}
Let $x,y\in I^{(0)}:=[0,1]$. Let $(n_k)_k$ be the sequence of accelerating times of the doubly positive acceleration and let $(I_{n_k})_{k\in \mathbb{N}}$  be the corresponding sequence of inducing intervals, so that $\mathcal{R}^k(T)$ is obtained renormalizing the induced map of $T$ on $I_{n_k}$. We denote $(\mathcal{P}_k)_k$ the associated sequence of dynamical partitions in towers over $I_{n_k}$.

\smallskip
We define the \emph{scale} of the interval $[x,y]$ as follows and show below that it is well defined.
\begin{definition}\label{def:scale}
The \emph{scale} $k=k(x,y)$ of the interval $[x,y]$ with endpoints $x,y$ is the minimum  $k\in \mathbb{N}$ such that the interval $[x,y]$ contains at least one full floor of the partition $\mathcal{P}_k$, i.e.~there exists an atom $F$ of the partition $\mathcal{P}_k$ such that $F\subset [x,y]$.
\end{definition}
\noindent To see that $k(x,y)$ is well defined, we should show that there exists a such $k$. To see this, recall that since $T$ is linearizable (i.e.~topologically conjugated to a standard IET, namely $T_0$), then $\lim_{k\to\infty}\mesh(\mathcal{P}_k)=0$ (see  Remark~\ref{meshtozero} in \S~\ref{sec:renormalisation}). Therefore, if we choose $k$ sufficiently large so that $\mesh(\mathcal{P}_k)<|x-y|/2$, $[x,y]$ must contain a floor of $\mathcal{P}_k$. This shows that $k(x,y)<+\infty$. We record the following immediate consequence of the definition of $k(x,y)$ as a Lemma, since it will be important later.

\begin{lemma}\label{rk:order}
If $k_0=k(x,y)$ is the scale of  $[x,y]\subset [0,1]$, then $[x,y]$ is contained in the union of  \emph{at most two} floors of the partition $\mathcal{P}_{k_0}$.  When the floors are two, i.e.~$[x,y]\subset F_1\cup F_2$, where  $F_1,F_2$ are floors of $\mathcal{P}_{k(x,y)-1}$, $F_1$ and $F_2$ are \emph{adjiacent} in $[0,1]$, i.e.~they share a common endpoint.
\end{lemma}
\begin{proof} Either no endpoints of $\mathcal{P}_k$ belongs to the interior of $[x,y]$, in which case $[x,y]$ is fully contained in \emph{one} of the floors of $\mathcal{P}_{k_0-1}$, or $[x,y]$ contains at least  an endpoint, call it $f$. In this case, though, there cannot be \emph{two} endpoints of $\mathcal{P}_{k_0-1}$ in the interior of $[x,y]$, otherwise $[x,y]$ would contain a floor of  $\mathcal{P}_{k_0-1}$ contradiction the minimality in the definition of $k_0=k(x,y)$. Then, if $f$ is the common endpoint of some floors $F_1$ and $F_2$ of $\mathcal{P}_{k_0-1}$, we conclude that the other endpoints are not in $[x,y]$, so $[x,y]\subset F_1\cup F_2$.
\end{proof}

The scale $k(x,y)$ allows first of all to estimate the size $|x-y|$ of $[x,y]$ through renormalisation as follows. Notice that in the following Lemma we will use the assumption that $T$ and $T_0$ are differentiably conjugate, as well as the Roth-type growth which we assume for $T_0$ (recall Definition~\ref{def:DC}).

\begin{lemma}[Estimate of the interval length through renormalisation]\label{lemma:denom}
There exists $c>0$ and $0<\rho_1<1$ such that every $x,y\in [0,1]$ 
$$
|x-y|\geq c \rho_1^{k(x,y)}.
$$
\end{lemma}
\begin{proof}
By definition of $k(x,y)$, $[x,y]$ contain some $F$ floor of $\mathcal{P}_{k(x,y)}$. Therefore, 
$[x,y]\geq |F|$. Let $F_0 = h(F)$ be the image of the floor $F$ under the conjugacy between $T$ and $T_0$. Then $F_0$ is a floor of the dynamical partition for the standard IET $T_0$. Since $h$ is a  $\mathcal{C}^1$-diffeomorphism, the ratio $|F|/|F_0|$ (by $\Vert| Dh \Vert|_{\infty}$ and its inverse). 
It is therefore sufficient to estimate $|F_0|$ and to show that $F_0 \geq c_0 \rho_1^{k(x,y)}$ for some $c_0>0$ and $\rho_1<1$ to conclude. Let $(I_0)_k^j$ denote the base floors of the Rohlin towers for $T_0$ after $k$ steps of renormalisation and $(I_0)_k=\cup_{j=1}^d (I_0)_k^j$ their union, which is the $k^{th}$ inducing interval for $T_0$. Then, since $T_0$ is a standard IET, if $F_0$ belongs to the tower over $(I_0)_k^{j_0}$, $|F_0|=|(I_0)_k^j|$. Thus, since the matrices $(A_n)_n$ are positive, 
applying \eqref{lenghts_lowbound} (with $n=k$ and $m=0$, so $|(I_0)_0|=1$), we get 
$$
|F_0|= |(I_0)_k^{j_0}|\geq \min_{1\leq j\leq d}|(I_0)_k^{j}|\geq  
 \frac{|1}{d \Vert A(0,k) \Vert}.
$$
Thus, exploting Property $(E)$ as well as property $(S)$ of the PTG (see Definition~\ref{def:DC}), we get that for every $\epsilon>0$ there exist $C_\epsilon$ such that
$$
|F_0|\geq  \frac{|(I_0)_0|}{d \Vert A_1^t \cdots A_k^t \Vert \Vert A_k \Vert}\geq \frac{1}{d\, C_\epsilon\, e^{\epsilon k } \Vert (A_k \cdot A_1)^t\Vert} \geq \frac{1}{d\, K C_\epsilon\, \rho^k e^{\epsilon k }}.
$$
The above estimate gives the desired exponential decay estimate for any choice of $\rho_1 $ such that $0<\rho_1< 1/\rho$.
\end{proof}

\subsubsection{Estimates on the variation when $x,y$ belong to the same orbit.}\label{sec:sameorbit}
In this subsection we  estimate the variation $|\varphi(x)-\varphi(y)|$ of the function $\varphi=\log Dh$ in a special   case, assuming that $x,y$ both belong to the orbit $\mathcal{O}_T(0)$ of $0$. This special case will then be used in the next \S~\ref{sec:appr} to estimate (by approximation) the general case of generic $x,y$.  
\smallskip

Let $z_0:=0$ and let us denote by $z_\ell:= T_\ell(0)$ the points in the orbit $\mathcal{O}_T(z_0):=\{ T_\ell z_0, l\in \mathbb{N}\}$ of $z_0=0$ under $T$. We then assume that $x:=z_p$ and $y:=z_q$ for some $p,q\in\mathbb{N}$.
We will also assume without loss of generality that $q>p$. Let us define
\begin{equation}\label{def:N}
N_k(p,q):= Card \{ \ell : \ p\leq \ell <q \ \text{and}\  z_\ell \in I_k\}
\end{equation}
to be the number of intersections of the orbit segment  $\{ z_\ell, p\leq \ell <p\}$ with $I_k$.
\begin{proposition}[Variation of $\varphi$ for points in the orbit of zero]\label{prop:specialvariation}
Assume that $T$ satisfies the Diophantine-like condition and the exponential convergence of renormalisation $d_{\mathcal{C}^2}(\mathcal{R}^k T, \mathcal{I}_d)\leq C\rho^k$ holds. Then, for any two points $z_p,z_q\in \mathcal{O}_T(0)$, if $z_p, z_q$ belong to the same floor of the partition $\mathcal{P}_{k}$, we have
$$
|\varphi(z_p)-\varphi(z_q)|\leq C  \, N_k (p,q)\,
 \rho^{k}, 
$$ 
where $N_k(p,q)$ is as in \eqref{def:N}. 
\end{proposition}
We prove this Proposition at the end of this subsection. 
As we explained in the outline in \S~\ref{sec:strategy}, the reason why we consider points in the same orbit first is that, in this case, setting $f:=\log DT$, in view of the relation $\varphi\circ T- \varphi = f$, we can reduce the study of the variation $|\varphi(z_p)-\varphi(z_q)|$ to the study of the Birkhoff sum $S_{q-p} \, f (z_p)$. We will decompose this Birkhoff sum into a number of special Birkhoff sums, plus two \emph{reminders}. The reminders have the following special form, that we call \emph{broken Birkhoff sum along a tower} (see Figure~\ref{fig:brokenBS} and the explanation after Definition~\ref{def:brokenBS}).

\begin{definition}[Broken Birkhoff sum along a tower]\label{def:brokenBS}
Let $\mathcal{P}_k^j$ be the Rohlin tower with base $I_k^j$ and height $q_k^j$. A broken Birkhoff sum along the tower $\mathcal{P}_k^j$ is sum  $B_k(x,y,m)$ of the form
$$
B_k(x,y,m):= S_m f(x) + S_{q_k^j -m + 1} f(T^m y), \qquad \text{where}\ \  x,y \in I^j_k , \ \ 0\leq m< q_k^j .
$$
\end{definition}
\noindent Thus, $\mathcal{P}_k^j$ and the sum is over two orbit segments, namely 
$$ \{x, T x, \dots, T^{m-1}x\}, \qquad \text{and} \qquad \{T^m y, T^{m+1} y, \dots, T^{q_k^j-1}y \}; 
$$
the first is the orbit of  a point $x$ in the base of the tower, 
 up the level $m-1$, while the second  the orbit 
of a point 
$T^m y$ belonging to the successive  level $m$ of the tower, up to the top of the tower. This motives the choice of the name \emph{broken Birkhoff sum along a tower}\footnote{In comparison to  \emph{Birkhoff sums along the tower}, which are sums over the orbit segment starting from a point in the base, up to the top of the tower, these sums also sums \emph{along a tower}, i.e.~sums over points which each belong to one and only floor of the tower, but are \emph{broken} in two parts, each of which is a sum over an orbit, see Figure~\ref{fig:brokenBS}.}. 
Equivalently, using the notation $S_{-m}\, f$ for negative Birkhoff sums introduced in \S~\ref{sec:renormalisation}, if $u:= T^{m} x$ and  $v:= T^m y$, we can write the sum $B_k(x,y,m)$ (recalling \eqref{backwardsBS}) 
also as
$$ 
S_{-m} f(u) + S_{q_k^j -m+1} f(v), \qquad \text{where}\ \   0\leq m< q_k^j \ \ \text{and}\ \ u,v\in T^m (I_k^j).
$$


\begin{figure}[!h]
	\begin{center}
		\def\svgwidth{0.6 \columnwidth}
\begingroup%
  \makeatletter%
  \providecommand\color[2][]{%
    \errmessage{(Inkscape) Color is used for the text in Inkscape, but the package 'color.sty' is not loaded}%
    \renewcommand\color[2][]{}%
  }%
  \providecommand\transparent[1]{%
    \errmessage{(Inkscape) Transparency is used (non-zero) for the text in Inkscape, but the package 'transparent.sty' is not loaded}%
    \renewcommand\transparent[1]{}%
  }%
  \providecommand\rotatebox[2]{#2}%
  \newcommand*\fsize{\dimexpr\f@size pt\relax}%
  \newcommand*\lineheight[1]{\fontsize{\fsize}{#1\fsize}\selectfont}%
  \ifx\svgwidth\undefined%
    \setlength{\unitlength}{445.7602521bp}%
    \ifx\svgscale\undefined%
      \relax%
    \else%
      \setlength{\unitlength}{\unitlength * \real{\svgscale}}%
    \fi%
  \else%
    \setlength{\unitlength}{\svgwidth}%
  \fi%
  \global\let\svgwidth\undefined%
  \global\let\svgscale\undefined%
  \makeatother%
  \begin{picture}(1,1.16864829)%
    \lineheight{1}%
    \setlength\tabcolsep{0pt}%
    \put(0,0){\includegraphics[width=\unitlength,page=1]{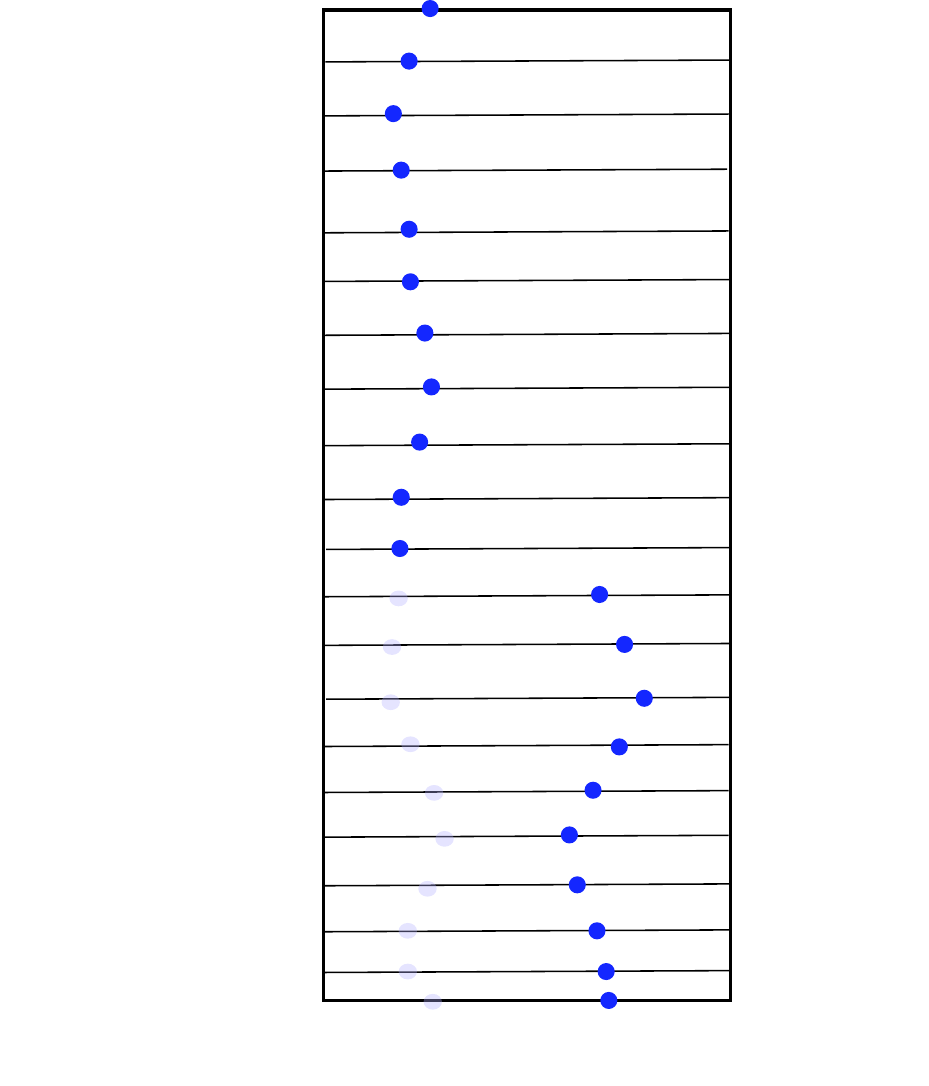}}%
    \put(0.39726291,0.01991807){\color[rgb]{0,0,0}\makebox(0,0)[lt]{\lineheight{1.25}\smash{\begin{tabular}[t]{l}$x$\end{tabular}}}}%
    \put(0.63166831,0.01177983){\color[rgb]{0,0,0}\makebox(0,0)[lt]{\lineheight{1.25}\smash{\begin{tabular}[t]{l}$y$\end{tabular}}}}%
    \put(0,0){\includegraphics[width=\unitlength,page=2]{brokenBS.pdf}}%
    \put(-0.00183994,0.83249097){\color[rgb]{0,0,0}\makebox(0,0)[lt]{\lineheight{1.25}\smash{\begin{tabular}[t]{l}$q^j_k - 1 - m$\end{tabular}}}}%
    \put(0.90053735,0.32391293){\color[rgb]{0,0,0}\makebox(0,0)[lt]{\lineheight{1.25}\smash{\begin{tabular}[t]{l}$m$\end{tabular}}}}%
  \end{picture}%
\endgroup%

	\end{center}
	\caption{A broken Birkhoff sum}
			\label{fig:brokenBS}
\end{figure}

\noindent We then define the supremum $R^j_k$ over all broken Birkhoff sums contained in the tower $\mathcal{P}_k^j$ (i.e.~over all choices 
 of level $m$  of a floor $F$ of $\mathcal{P}^j_k$ and of points $u,v \in F$): 
\begin{equation}\label{def:R}
R_k^j:= 
\sup_{0\leq m<q_k^j}\ \  \sup_{u,v \in T^m (I_k^j)} \left( S_{-m+1} f (u) + S_{q_k^{i}+1-m} f (v) \right).
\end{equation}

The following Lemma~\ref{lemma:basicestimate} provides the basic estimate that we will use 
 to estimate  Birkhoff sums of $f=\log DT$ 
in terms of special Birkhoff sums and broken Birkhoff sums of level $k$. 
\begin{lemma}[Birkhoff sums decomposition at level $k$]\label{lemma:basicestimate}
If $z_p, z_q \in \mathcal{O}_T(0)$ with
 $ p<q$. If $k$ is such that both $x_p$ and $x_q$ belong to the same floor $F$ of some tower $\mathcal{P}_{k}^j$ of  $\mathcal{P}_{k}$ and Let $N_k(p,q)$ is, as in \eqref{def:N},  
 the number of intersections between the orbit segment $\{ z_p, z_{p+1}, \dots, z_{q-1}\}$ and $I_k$, we have
\begin{equation}\label{specialestimate}
| S_{q-p} \, f (z_p)  | \leq N_k(p,q)\, \Vert f_k \Vert_{\infty} + R_k^j (f), 
\end{equation}
where 
$R^j_k(f)$ is the supremum over broken Birkhoff sums contained in $\mathcal{P}^j_k$, see \eqref{def:R}.
\end{lemma}
\begin{proof}
To keep the notation simple, let us write $N:=N_k(p,q)$. Let $\ell_1<\ell_2<\dots<\ell_N$ the be labels of the $N$ intersections of the orbit segment $\{ z_p, z_{p+1}, \dots, z_{q-1}\}$ and $I_k$. 
 Thus, $z_{\ell_1}$ is the first intersection of the forward orbit $\mathcal{O}_T(z_p)$ with $I_k$,  for any $1\leq n\leq N-1$, $z_{\ell_{n+1}}=T_k (z_{\ell_n})$ are the successive returns and $z_{\ell_N}$ is the last intersection before $z_q$. Thus, we can  decompose the Birkhoff sum $S_{q-p} \, f (z_p)$ as
\begin{equation}\label{1stepdecomp}
 S_{q-p} \, f (z_p) =  S_{\ell_1-p} f (z_{p}) + \sum_{n=1}^{N-1} S_{\ell_{n+1}-\ell_n}f (z_{\ell_n}) +S_{q-\ell_{N}} f (z_{\ell_N}).
\end{equation}
Notice now that for any $1\leq n<N$, $\ell_{n+1}-\ell_n$ is exactly the height $q_k^{j(n)}$ of the tower $\mathcal{P}_k^{j(n)} $ which contains $z_{\ell_n}$ in its base.  Thus, recalling that, by definition the special Birkhoff sum (see \S~\ref{sec:renormalisation}), $f_k(z_\ell)=S_{q_{k}^{j(n)}} f (z_{\ell_n})$ when $z_{\ell_n}\in I_k^{j(n)}$, we get
\begin{equation}\label{SBSestimate}
|\sum_{n=1}^{N-1} S_{\ell_{n+1}-\ell_n}f (z_{\ell_n})|= |\sum_{n=1}^{N-1} f_k (z_{\ell_n})|\leq N \Vert f_k\Vert_{\infty}.
\end{equation}
We claim that the first and the last term in \eqref{1stepdecomp} together form a broken Birkohff sum along the tower $\mathcal{P}_k^j$ and therefore can be estimated by the supremum $R_k(f) $ defined in \eqref{def:R}. Let $F= T^mI_k^j$ for some $0\leq m< q^j_k$ be the floor of $\mathcal{P}^j_k$ which, by assumption, contains both $x_p$ and $x_q$. 
Then it suffices to remark that, since  $z_p \in  T^{m}I_k^{i} $, 
we have that $\ell_1-p = q^{i}_k+1-m$ (which is the number of floors of the tower  $\mathcal{P}_k^{i}$ \emph{above} the floor which contains $z_p$) and, respectively, since we also have that $ z_q \in  T^{m}I_k^{i} $, 
then the previous (and last) visit $z_{\ell_n}$ to $I_k$ was the projection of $z_q$ to the base $I_k^{i}$, i.e.~
$\ell_N= q- m$. Thus, $S_{q-\ell_{N}} f (z_{\ell_N})=S_{m-1} f (z_{\ell_N})$ and furthermore, since $z_q=T^{m-1}(z_{\ell_n})$,  we can  write
$S_{m-1} f (z_{\ell_N})= S_{-m} f(z_q)$ (by \eqref{backwardsBS} in \S~\ref{sec:renormalisation}).
 Thus
$$| S_{\ell_1-p} f (z_{p}) +S_{q-\ell_{N}} f (z_{\ell_N})|  = 
| S_{q_k^{i}+1-m} f (z_p)  + S_{-m} f(z_q) |\leq  R^j_k(f), 
$$
where the last estimate (since $z_p, z_q\in T^m(I_k)^j$ and hence the two sums form a broken Birkhoff sum along the tower $\mathcal{P}_k^j$) follows from the definition \eqref{def:R} of $R^j_k(f)$. 
Combining this estimate with \eqref{SBSestimate}, we now see that \eqref{1stepdecomp} can be estimated by triangle inequality as claimed.
\end{proof} 

The norms $\Vert f_k \Vert_{\infty} $ in \eqref{specialestimate} decay exponentially when there is exponential convergence of renormalisation in view of Lemma~\ref{SBSviaR}. Let us show now that also the term $R^j_k(f)$ that controls broken Birkoff of level $k $ can be shown to decay exponentially in $k$, by exploiting exponential convergence of renormalisation (in particular convergence in $d_{\mathcal{C}^2}$ plays a crucial role  here). 


\begin{proposition}[Exponential estimate of broken special Birkhoff sums]\label{prop:reminderestimate}
Under the exponential convergence of renormalisation assumption that $d_{\mathcal{C}^2 }(\mathcal{R}^k T, \mathcal{I}_d)\leq C \rho^k$,  for every $1\leq j\leq d$, we have that there exists $\rho'$
$$
\sup_{z\in I_k} |R_k^j(f)-f_k (z)| \leq C \rho'^k, \qquad \text{for}\ \text{every} \ k\in \mathbb{N}.
$$
Thus, in particular $|R_k^j(f)|\leq \Vert f_k\Vert_\infty + C \rho'^k$ for every  $k\in \mathbb{N}$. 
\end{proposition}

\begin{proof}
Choose 
any $0\leq m <q^j_k$ and consider any two points $u,v\in T^m I_k^j$. 
Let $x,y$ be the projections to the base of the tower of $u,v$ respectively, namely $x=T^{-m}u, y:=T^{-m}v$ (see Figure~\ref{fig:brokenBS}, left).  Choose $z\in I_k^j$ which belong to the interval with endpoints $x$ and $y$. 
Assume without loss of generality that $x<z<y$. We want to compare the broken Birkhoff sum 
$B_k (u,v,m):= S_{-m} f(u) + S_{q_k^j -m+1} f(v)$ 
with the special Birkhoff sum $f_k(z)$ (looking at Figure~\ref{fig:brokenBS} may help the reader to follow). Then, using the definition of broken Birkhoff sums and special Birkhoff sums and rearranging the terms,
\begin{align}
B_k (u,v,m) - f_k(z)  = \nonumber 
& \left(\sum_{\ell=0}^{m-1}f( T^\ell y) +\sum_{\ell=0}^{q_k^{j}+1-m} T^\ell f(T^m x)\right) -  \sum_{\ell=0}^{q_k^j-1} f(T^\ell z) \\
 =& \nonumber
\left(\sum_{\ell=0}^{m-1}f( T^\ell y) + \sum_{\ell=0}^{q_k^{j}+1-m} T^\ell f(T^m x) \right)- \left (\sum_{\ell=0}^{m-1} f(T^\ell z) - \sum_{\ell=0}^{q_k^{j}+1-m}  f(T^\ell z) \right) \\
 =&
 \sum_{\ell=0}^{m-1}\left( f (T^\ell y)- f( T^\ell z ) \right) +  \sum_{\ell=0}^{q_k^j-1} \left( f( T^\ell z)-f (T^\ell x) \right).\label{differences}
\end{align}
We now want to exploit that,  for any interval $[a,b]$, since $f=\log DT$ and $\eta_T=D\log DT$,  we have that $f(b)-f(a) = \int_{[a,b]} \eta_T(\chi) \mathrm{d}\chi$. Remark  that the intervals involved in the sum, namely
$$\{ T^\ell\, [y,z], \ 0\leq \ell <m\}, \quad \text{and} \quad \{  T^\ell\, [y,z], \ 0\leq \ell <m\}$$ 
 are pairwise disjoint (since they belong to distinct floors of the Rohlin tower $\mathcal{P}_k^j$). We will now  group  these integrals into blocks which correspond to Rohlin towers that can be controlled by $|N|(\mathcal{R}^n T)$. 

Let us first estimate the first sum in \eqref{differences}. The second sum in \eqref{differences} can be treated in an analogous way. Recalling the geometric decomposition of a Birkhoff sum $S_m f $  into special Birkhoff sums (see \S~\ref{sec:renormalisation}),  for any $u\in [y,z]$, since $0\leq m<q_k^j$ we can write
$$\sum_{\ell=0}^{m-1} f (T^\ell u) = \sum_{n=0}^{k} \sum_{i=0}^{a_n-1} \left( f_n ( y_n^i) - f_n ( z_n^i) \right), 
$$
where $|a_n|\leq \Vert A_n\Vert $ and $u_n^i\in I_n$ for each $1\leq i\leq a_0$.  Let $[y_n^i, z_n^i]$ denote the image of $[y,z]$ under an iterate of $T$ to which $u_n^i$ belong.
 Since $f=\log DT$ and therefore $f_n= \log D T_n$ 
$$
f_n ( y_n^i) - f_n ( z_n^i) =   \int_{y_n^i}^{z_n^i}D \log D T_n (\chi) \mathrm{d} \chi = \int_{y_n^i}^{z_n^i} \eta_{T_n} (\chi) \mathrm{d} \chi.
$$
Thus, since $[y_n^i, z_n^i]$ is a subinterval of the inducing interval $I_n$, by the assumptions on $\mathcal{C}^2$-exponential convergence and Lemma~\ref{nLdecay}  we get that 
$$
\left| \sum_{\ell=0}^{m-1}\left( f (T^\ell y)- f( T^\ell z ) \right)\right| =  \left| \sum_{n=0}^{k}\sum_{i=0}^{a_n-1} \int_{y_n^i}^{z_n^i} \eta_{T_n} (\chi) \mathrm{d} \chi  \right| \leq  \sum_{n=0}^{k}\sum_{i=0}^{a_n-1} \left|\int_{y_n^i}^{z_n^i} \eta_{T_n} (\chi) \mathrm{d} \chi  \right|
$$

\noindent
 We now estimate $\left|\int_{y_j^i}^{z_j^i} \eta_{T_j} (\chi) \mathrm{d} \chi  \right|$.
 Since $\eta_{T_j} = {\eta_{\mathcal{R}^jT}}/{|I_j|}$ and $[y_j^i, z_j^i]\subset I_k$, we get 
\begin{equation}\label{toestimate} \left|\int_{y_j^i}^{z_j^i} \eta_{T_j} (\chi) \mathrm{d} \chi  \right| \leq \frac{|z_j^i - y_j^i|}{|I_j|} || \eta_{\mathcal{R}^jT} ||_{\infty} \leq  \frac{|I_k|}{|I_j|} || \eta_{\mathcal{R}^jT} ||_{\infty} .
\end{equation}
Since $\Vert \eta_{\mathcal{R}^jT} ||_{\infty} $ descreases exponentially by Lemma \ref{nLdecay} and the sequence $(|I_j|)_j$ is decreasing so that since $j\leq k$ we have $|I_k|\leq |I_j|$, this immediately gives 
$$ \left|\int_{y_j^i}^{z_j^i} \eta_{T_j} (\chi) \mathrm{d} \chi  \right| \leq C' \rho_2^j,$$
which will be used to estimate values of $j$ comparable with $k$. To estimate small values of $j$, 
let us recall that the lenght vectors are transformed by the cocycle  by $\lambda_j = A^\dag(j,k) \lambda_k$, see \eqref{lengthsrelation},  so that, since the matrices $(A_{j})_j$ and hence also the matrices $(A_{j}^\dag)_j$ are positive (see Condition (P) of Definition \ref{def:DC}), the lenghts  $|I_j|=\Vert \lambda_j \Vert$ satisfy
$
|I_j|= \Vert   A(j,k)^\dag \lambda_k \Vert\geq d^{k-j} |I_k|
$. 
Thus, combining this  estimate with  Lemma \ref{nLdecay} to estimate \eqref{toestimate},  we obtain that 
$$ \left|\int_{y_j^i}^{z_j^i} \eta_{T_j} (\chi) \mathrm{d} \chi  \right| \leq C' \lambda_1^{k- j}{\lambda_2}^j.$$
where $C' > 0$, $\lambda_2<1$ and $\lambda_1<1/d$,  

\vspace{2mm}
We now consider the sum $\sum_{j=0}^{k}\sum_{i=0}^{a_j-1} \left|\int_{y_j^i}^{z_j^i} \eta_{T_j} (\chi) \mathrm{d} \chi  \right|$. Since  $a_j\leq \Vert A_j\Vert $  (see \S~\ref{sec:renormalisation}) and  the matrices norms $\Vert A_j\Vert $ grow subexponentially by (S) of Definition \ref{def:DC},  for any $\lambda_3<\lambda_2$, for some $C_0>0$ we have
$$\sum_{j=0}^{k}\sum_{i=0}^{a_j-1} \left|\int_{y_j^i}^{z_j^i} \eta_{T_j} (\chi) \mathrm{d} \chi  \right|\leq C'  \sum_{j=0}^{k} a_j\, \lambda_1^{k- j}{\lambda_2}^j\leq C_0  \sum_{j=0}^{k}   \lambda_1^{k- j}{\lambda_3}^j. $$

We now consider two cases. If $j\geq k/2$, then, since $\lambda_1<1$,  $\lambda_1^{k- j}{\lambda_3}^j\leq {\lambda_3}^j\leq  \lambda_3^{k/2}. $
 If $j\leq k/2$, then, since $\lambda_3<1$,  $\lambda_1^{k- j}{\lambda_3}^j\leq\lambda_1^{k- j}\leq  \lambda_1^{k/2} $. Thus,  taking $\lambda:=  \max \{\lambda_1, \lambda_3 \}$, we can estimate
 $$
  \sum_{j=0}^{k}  \lambda_1^{k- j}{\lambda_3}^j\leq \sum_{j=0}^{k}  \lambda^{k/2}\leq k  \lambda^{k/2} . $$
Since this term decreases exponentially,  we get the main statement of the Lemma.

\end{proof}

We now have all the ingredients to prove Proposition~\ref{prop:specialvariation}.
\begin{proof}[Proof of Proposition~\ref{prop:specialvariation}]
Assume WLOG that $q>p$ and let $N:= N_k(p,q)$ be the number of intersections of $\{  z_p, T z_p, \dots , z_{q-1}\} $ with $I_k$.  
Since $\varphi$ satisfies the cohomological equation $\varphi\circ T -\varphi = f$ where $f=\log DT$, a telescopic sum argument gives that, if $q>p$, $\varphi(z_q)-\varphi(z_p)=  S_{q-p} \, f (z_p)$. Thus,   combining the estimates given by the basic decomposition in Lemma~\ref{lemma:basicestimate} (whose assumptions hold since both $z_p$ and $z_q$ belong to the same floor of $\mathcal{P}_k$) and 
by Proposition~\ref{prop:reminderestimate} (which can be applied since we assume that $d_{\mathcal{C}^2} (\mathcal{R}^k T, \mathcal{I}_d)\leq C \lambda^k$),  gives that 
 we have that
$$
|\varphi(z_p)-\varphi(z_q)|= | S_{q-p} \, f (z_p)  | \leq N \Vert f_k \Vert_{\infty} + |R_k^j(f)| \leq (N+1) \Vert f_k \Vert_{\infty} + C\lambda^k. 
$$
Since, by the convergence of renormalisation assumption, in view of Lemma~\ref{SBSviaR}, also $\Vert f_k \Vert_{\infty}$ decay exponentially in $k$, we get the desired conclusion.  
\end{proof}

\subsubsection{General case via single orbit approximations.}\label{sec:appr}
We will now build approximations of any pair $x,y$ of points in $[0,1]$ though sequences of points in $\mathcal{O}_T(0)$. Our approximations will have the properties listed in the following Proposition. An illustration of the approximations and the relative positions of the points is given in Figure~\ref{fig:appr}.


\begin{figure}[!h]
	\begin{center}
		\def\svgwidth{0.8 \columnwidth}
			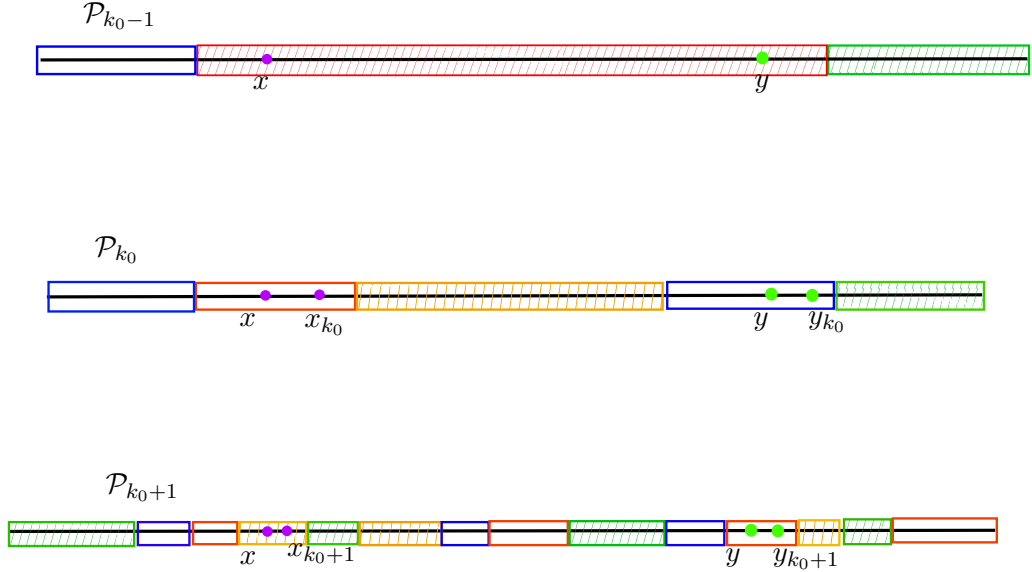
	\end{center}
	\caption{Schematic representation of the approximating sequences.}
			\label{fig:appr}
\end{figure}

\begin{proposition}[Fixed scale single orbit approximation.]\label{prop:appr}
Let $T_0$ be a Keane IET and let $T$ be topologically conjugate to $T_0$. 
Given any $x,y\in [0,1]$, assume that $k_0\in \mathbb{N}$ is such that $x$ and $y$ belong to the same floor $F$ of $\mathcal{P}_{k_0-1}$ but distinct floors of $\mathcal{P}_{k_0}$. Then  
there exists sequences $(x_k)_{k\in \N}$, $(y_k)_{k\in \N}$ such that:
\begin{itemize}
\item[(i)] all points $x_k, y_k$, $k\in \mathbb{N}$ belong to the orbit $\mathcal{O}_T(0)$ of $0$;
\item[(ii)] the sequences approximate $x$ and $y$, i.e.~$\lim_{k\to \infty} x_k = x$ and $\lim_{k\to \infty} y_k = y$;
\item[(iii)] for any $k\geq k_0$, $x_k$ belongs to the floor of $\mathcal{P}_{k}$ which contains $x$ and  $y_k$  to the floor of $\mathcal{P}_{k}$ which contains $y$;
\item[(iv)]   the number of intersections of the orbit segments which join $x_{k_0}$ with $y_{k_0}$ with $I_{k_0-1}$ is at most $2 \Vert A_{k_0-1}\Vert \Vert A_{k_0}\Vert$; 
\item[(v)] for any $k\geq k_0$, the number of intersections of the orbit segments which join $x_k$ with $x_{k+1}$ and $y_k$ and $y_{k+1}$ with $I_{k}$ is at most $2 \Vert A_k\Vert \Vert A_{k+1}\Vert$.
\end{itemize}
\end{proposition}
\noindent The Proposition is proved shortly below. Let us first state and prove an auxiliary Lemma which will allow to build the approximations inductively.

\begin{lemma}[Inductive step]\label{lemma:inductivestep} Under the assumption of Proposition~\ref{prop:appr}, 
given a point $z_p\in \mathcal{O}_T(0)$, let $ F_k$ be the floor of $\mathcal{P}_k$ such that $z_p\in F_k$. Choose any floor $F_{k+1}$
 in  $\mathcal{P}_{k+1}$ such that $F_{k+1}\subset F_k$. Then we can find another point $z_q\in \mathcal{O}_T(0)$ with $q>p$ such that $z_q\in F_{k+1}$ and  the number of intersections of $\{ z_p, z_{p+1}, \dots, z_{q-1}\}$ with $I_{k}$ is at most $2 \Vert A_k\Vert \Vert A_{k+1}\Vert $.
\end{lemma}
\begin{proof}
Fix $F_{k+1}\subset F_k$ as in the statement. Consider an orbit segment $\{ z_p, T z_p , \dots, T^{n-1} z_p\}$ of length $n$. We claim that if $n:=2\max_l q_{k+2}^l$ there exists a point $z_q$ with $p\leq q < p+n$ such that $z_q\in F_{k+1}$.
 Let  $i, j \in \{1,\dots, d\}$ be the two indexes such that, respectively, $F_{k+1}\in \mathcal{P}_k^i$ and $F_{k+2}\in \mathcal{P}_k^j$. Notice that, by the dynamics in a Rohlin towers, if a point belongs to $ I_{k+1}^j$, then its orbit up to time $q_{k+1}^j$ enters $F_{k+1}$. Moreover, since  $z_p\in F_{k+2}$, the orbit of $x$ up to time $2\max_l q_{k+2}^l$ moves across at least one full tower of $\mathcal{P}_{k+2}$ (since in time at most $q_{k+2}^j$, $z_p$ reaches the top of the tower $\mathcal{P}_{k+1}^j$ to which it belongs and then it has time to cross from bottom to top the next tower of $\mathcal{P}_{k+1}$ which it visits). 
 
 Since we are considering a positive acceleration, $A_{k+2}>0$, which (by the dynamical interpretation of the entries, see \S~\ref{sec:renormalisation}) means that, in each tower of $\mathcal{P}_{k+2}$, there are levels which belong to $I_{k+1}^i$. Thus, the orbit of $z_p$ up to time $n$ will enter $I_{k+1}^i$ as well as each floor of the tower $\mathcal{P}_{k+1}^i$,  so in particular it will visit $F_{k+1}\subset \mathcal{P}_{k+1}^i$. If $m$ is then the index $p\leq m<p+n$ such that $T^m (x_p) \in F_{k+1} $, we set $q:= p+m$. 


\begin{figure}[!h]
	\begin{center}
		\def\svgwidth{0.6 \columnwidth}
			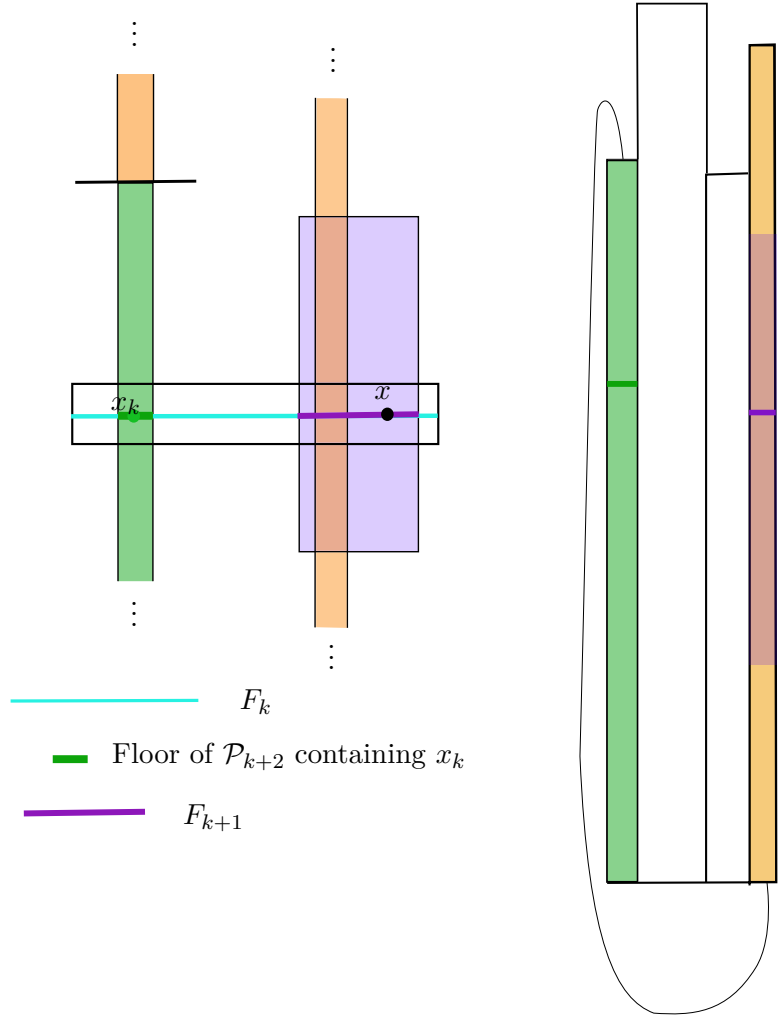
	\end{center}
\caption{Inductive step of the approximation construction (picture typo: $F_{k+2}$ should have been green small floor which contains $x_k$). The index $j$ is the light green color (of the tower of level $k+2$ which contains $x_k$ in figure), the index $i$ is the light blue color of level $k+1$ which contains the chosen floor of level $k+1$ in figure).  In time $n$, $x$ reaches the top of the green tower to which it belongs and goes through all of the next tower of level $k+2$, orange in figure.  So its orbits runs through all of the light blue tower and hence enters the desired floor before time $n$.}
			\label{fig:step}
\end{figure}

To conclude, we only need to estimate the cardinality of $\{ z_p, z_{p+1}, \dots, z_{q-1}\}\cap I_{k}$. 
Notice first of all that, 
 by the above construction, the orbit of $z_p$ of length $q-p$ crosses at most two towers of $\mathcal{P}_{k+2}$; more precisely  $q-p\leq q_{k+2}^j+ q_{k+2}^l$ where $q_{k+2}^j$ is the height of the tower  $\mathcal{P}_{k+2}^j$ which contains $x$ and $q_{k+2}^l$ is the height of the successive tower of  $\mathcal{P}_{k+2}$ which contains the orbit $\mathcal{O}_T(x_p)$ (in particular it contains the first visit of $\mathcal{O}_T(x_p)$  to $I_{k+2}$). 
Therefore, the number of visits we want to estimate is bounded above by twice the number of floors of towers of $\mathcal{P}_{k+2}$ which belong to $I_k$. This, by the dynamical interpretation of the cocycle matrices entries, is estimated by twice the maximum  norm of the columns of $A_{k+1}\, A_k $ (see \S~\ref{sec:renormalisation} and recall also that, by \eqref{heightsrelation},   $q_{k+1}= A_{k+1}\, A_k\, q_k$)
 and therefore is bounded by $2 \Vert A_{k+1} A_{k} \Vert\leq  2 \Vert A_{k+1} \Vert \Vert A_{k} \Vert$. This gives the desired estimate  hence concludes the proof.
\end{proof}

We can now prove by induction the existence of the approximation sequences in Proposition~\ref{prop:appr}.
\begin{proof}[Proof of Prop.~\ref{prop:appr}] 
Let $F$ be the floor of $\mathcal{P}_{k_0-1}$ which by assumption contains both $x$ and $y$ and let us denote by $F(x)$ and $F(y)$ respectively the two floors of  $\mathcal{P}_{k_0}$ which contain $x$ and $y$ respectively. 
Since $T$ is minimal (being topologically conjugate to a Keane IET, which is minimal by \cite{Ke:int}, see \S~\ref{giet}), the orbit $\mathcal{O}_T(0)$ is dense. Let $p_0$ be such that $z_{p_0}:=T^{p_0}(0)\in F(x)$. 

\smallskip
\noindent {\it First $k_0$ steps of the decomposition.} Let us first define the points $x_{k_0}, y_{k_0}$ in the decomposition. Let 
 Applying Lemma~\ref{lemma:inductivestep} twice (taking both times $z_p:=z_{p_0}$ and $F_{k}=F$, but taking $F_{k+1}$  to be $F(x)$ or $F(y)$ respectively), we find two points in $\mathcal{O}_T(z_{p_0})\subset \mathcal{O}_T(0)$, that we will call $x_{k_0}:= T^{i_0}(0)$ and  $y_{k_0}:= T^{j_0}(x)$ respectively (where by construction $i_0,j_0\geq p_0$) such that $x_{k_0}\in F(x)$ and $y_{k_0}\in F(y)$ and the cardinality  of  the intersections with $I_{k_0-1}$ and the orbit segments 
$$
\{  T^{p_0}(0), T^{p_0+1}(0),  \dots, T^{i_0-1}(0)\} \quad \textrm{and} \quad  \{ T^{p_0}(0),  T^{p_0+1}(0),   \dots, T^{j_0-1}(0)\} 
$$
are both bounded above by $\Vert A_{k_0-1}\Vert\, \Vert A_{k_0}\Vert$. Therefore the same bound holds for the number of intersections between $I_{k_0-1}$ and the orbit segment between $x_{k_0}$ and $y_{k_0}$, which is contained in one of the above segments (since if  $i_0\geq j_0$, 
it has the form $\{ T^{j_0}(0), T^{j_0+1}(0), \dots, T^{i_0-1} 0\} $ and it is contained in the first segment above, or if $i_0<j_0$, it is similarly contained in the second). 
We can then set $x_k=x_{k_0}$ and  $y_k=y_{k_0}$ for all $0\leq k\leq k_0$. 
Notice that property $(iv)$, as well as property $(i)$ for $0\leq k\leq k_0$, so far hold by construction. 
Furthermore, property $(iii)$ for $0\leq k\leq k_0$ holds since it holds for $k=k_0$ and the partitions $(\mathcal{P}_n)_{n}$ are nested, see \S~\ref{sec:renormalisation} (so if $x_k=x_{k_0}$ belongs to the same floor $F_{k_0}$ of $\mathcal{P}_{k_0}$ which contain $x$, $x$ and $x_k$ also belong to the same floor of $\mathcal{P}_{k}$, since $k\leq k_0$).

\smallskip
\noindent {\it Successive steps $k>k_0$ of the decomposition.} We will now define $x_k$ and $y_k$ for $k> k_0$ inductively. The construction is identical for $(x_k)_{k>  k_0}$ and $(y_k)_{k > k_0}$, so we will do it only for  $(x_k)_{k>  k_0}$. Assume  we have already define $x_{k}$ for some $k\geq k_0$ and that, as part of inductive assumption, $x_{k}\in F_{k}$ where $F_{k}$ is the floor of $\mathcal{P}_{k}$ which contains $x$. Let $F_{k+1}$ be the floor of $\mathcal{P}_{k+1}$ which contains $x$. Clearly $F_{k+1}\subset F_{k}$. Thus, by applying Lemma~\ref{lemma:inductivestep} (taking as $z_p$ in the statement of the Lemma the point $x_{k}$), we find another point $x_{k+1}\in \mathcal{O}_T(x_k)\subset \mathcal{O}_T(0)$ such that the number of intersections of the orbit segment between $x_{k} $ and $x_{k+1}$ with $I_{k}$ is at most $\Vert A_k\Vert \Vert A_{k+1}\Vert$.  
Thus, this construction gives shows that property $(iii)$ and $(v)$, as well as property $(i)$ for the remaining $ k\geq k_0$, hold. 

\smallskip
\noindent {\it Approximation property.}  We are only left to check $(ii)$; this property though follows from the construction and 
 Remark~\ref{meshtozero}, that shows that, since $T$ is by assumption conjugate to a Keane IET (which is therefore minimal, see \S~\ref{giet}),  $\mesh(\mathcal{P}_k)$ goes to zero as $k$ grows; indeed, if we denote by $F_k(x)$ and respectively $F_k(y)$ the floors of $\mathcal{P}_k$ which contain respectively $x$ and $y$, from property $(iii)$ we have that
$$
|x_k-x|\leq |F_k(x)|\leq  \mesh(\mathcal{P}_k), \qquad |y_k-y|\leq |F_k(y)|\leq  \mesh(\mathcal{P}_k) $$
and therefore $0\leq \limsup_{k\to \infty}|x_k-x| \leq \lim_{k\to \infty} \mesh(\mathcal{P}_k)=0$, so $\lim_{k\to \infty}x_k=x$ and, analogously, $\lim_{k\to \infty}y_k=y$. This proves also $(ii)$ and therefore concludes the proof.
\end{proof}

We can now combine Proposition~\ref{prop:appr}, which provides approximation sequences, with what proved in the previous section in the special case of points in $\mathcal{O}_T(0)$ (namely Proposition~\ref{prop:specialvariation}), to estimate $|\varphi(x)-\varphi(y)|$ any $x,y\in [0,1]$ as follows. 
\begin{proposition}\label{prop:numerator}
There exists $C>0$ and $\lambda_2<1$ (with $\lambda_2>\lambda_1$ where $\lambda_1 $ appears in Lemma~\ref{lemma:denom}) such that,  
for any $0<x< y<1$, 
$$
|\varphi(x)-\varphi(y)|\leq C \lambda_2^{k(x,y)}.
$$
\end{proposition}
\begin{proof}  Let $k_0:=k(x,y)$ be the scale of $[x,y]$ (see Definition~\ref{def:scale}). By Lemma~\ref{rk:order}, there exists either at most two floors of $\mathcal{P}_{k_0-1}$ whose union contains $[x,y]$. Let us assume first that there is a \emph{unique} such floor $F \supset [x,y]$. The other case will be reduced to this at the very end of the proof. 
 
\smallskip
\noindent {\it Case $[x,y]\subset F$.} In this case, the assumptions of Proposition~\ref{prop:appr} hold, so the Proposition gives two sequences  $(x_n)_n, (y_n)_n$ approximating $x$ and $y$ and satisfying the properties $(i)-(v)$ of Proposition~\ref{prop:appr}.  Since $h$ and hence $\varphi:=\log DT$ are continuous and $\lim_k x_k= x$ and $\lim_k y_k= y$ (by property $(ii)$),
$$
|\varphi(x)-\varphi(y)|= \lim_{k\to \infty} |\varphi(x_k)-\varphi(y_k)|.
$$
Thus it suffices to show that for every $k$, $|\varphi(x_k)-\varphi(y_k)|\leq C_0 \lambda_2^{k_0}$ for some constant $C_0=C_0(x,y)>0$  and  exponenent $0<\lambda_2<1$ which are independent on $k$. To see this, we consider the two telescopic series
$$
\varphi(x_k) = \varphi({x_{k_0}})+ \sum_{k_0 \leq \ell <  k} \varphi(x_{\ell+1})-\varphi(x_\ell) , \qquad \varphi(y_k) = \varphi ({y_{k_0}})+ \sum_{k_0 \leq \ell <  k} \varphi(y_{\ell+1})-\varphi(y_\ell) 
$$
and therefore get the estimate 
$$
|\varphi(x_k)-\varphi(y_k)|\leq  \sum_{k_0 \leq \ell <  k} |\varphi(x_{\ell+1})-\varphi(x_\ell)| + |\varphi(x_{k_0})-\varphi(y_{k_0})| + \sum_{k_0\leq\ell < k} |\varphi(y_{\ell+1})-\varphi(y_\ell)|
$$
Let us estimate first the central term of this estimate, then the two series.  
For the central term, let us apply Proposition~\ref{prop:specialvariation} to   $x_{k_0}:=T^{i_0}0 $ and $y_{k_0}:=T^{j_0}0$, which by construction belong to $F$ (which by assumption in this case is the unique floor of $\mathcal{P}_{k_0-1}$ which contain $x$ and $y$ respectively). Since the quantity $N_{k_0-1}(i_0,j_0)$ (defined in \eqref{def:N}) satisfies $N_{k_0-1}(i_0,j_0)\leq 2 \Vert A_{k_0-1}\Vert \Vert A_{k_0}\Vert $ by property $(iv)$ of Proposition~\ref{prop:appr}, we  get that
$$
|\varphi(x_{k_0})-\varphi(y_{k_0})|\leq  C N_{k_0-1}(i_0,j_0)  \lambda^{k_0-1}\leq  2 C \Vert A_{k_0-1}\Vert \Vert A_{k_0}\Vert  \lambda^{k_0-1}.
$$
We will estimate now the series with terms $|\varphi(x_{\ell+1})-\varphi(x_\ell)|$. The series with terms of the form $|\varphi(y_{\ell+1})-\varphi(y_\ell)|$ can be estimated in an identical way. For each $k_0\leq \ell\leq k_0$, each term of the series can be estimated similarly to the central term above: since by property $(ii)$ $x_\ell$ and $x_{\ell+1}$ both belong to the same floor of $\mathcal{P}_{\ell}$ (namely the floor of $\mathcal{P}_{\ell}$ which contains $x$, directly by property $(ii)$ for $x_\ell$ and by property $(ii)$ together with the fact that $\mathcal{P}_{\ell+1}$ refines $\mathcal{P}_\ell$ in the case of $x_{\ell+1}$) and the number of intersections of the orbit segment between $x_\ell$ and $x_{\ell+1}$ with $\mathcal{P}_\ell$ is controlled by property $(v)$, we get by Proposition~\ref{prop:specialvariation} that
$$
|\varphi(x_\ell)-\varphi(x_{\ell+1})|\leq 2 C \Vert A_{\ell}\Vert \Vert A_{\ell+1}\Vert  \lambda^{\ell}.
$$
The differences $|\varphi(y_\ell)-\varphi(y_{\ell+1})|$ can be estimated in the same way. 

Combining the estimates of the central term and the series terms, and then recalling condition $(S)$ of Definition~\ref{def:DC}, we get
\begin{align*}
|\varphi(x_k)-\varphi(y_k)|& \leq 4 C \sum_{k_0\leq \ell \leq k}  \Vert A_{\ell}\Vert \Vert A_{\ell+1}\Vert  \lambda^{\ell}+2 C \Vert A_{k_0-1}\Vert \Vert A_{k_0}\Vert  \lambda^{k_0-1} \\ & \leq  4 C \sum_{k_0-1\leq \ell \leq k}  \Vert A_{\ell}\Vert \Vert A_{\ell+1}\Vert  \lambda^{\ell} \leq 4 C  (C_\epsilon)^2 \sum_{k_0-1\leq \ell \leq k} e^{\ell \epsilon} e^{(\ell+1)\epsilon}  \lambda^{\ell}.
\end{align*}
Since for any choice of $\lambda_2$ such that $ \lambda\leq \lambda_2<1$, choosing $\epsilon>0$ sufficiently small, this expression controlled by the tail $\sum_{\ell \geq k_0-1}\lambda_2^\ell$ of geometric series of step $\lambda_2<1$.  
Since this is a geometric series and recalling that $k_0=k(x,y)$, this gives the desired estimate.
\smallskip

\noindent {\it Case $[x,y]\subset F_1\cup F_2$.} We consider now the case in which there are \emph{two} floors $F_1, F_2$ of $\mathcal{P}_{k_0-1}$ such that $[x,y]\subset F_1\cup F_2$. In this case $F_1$ and $F_2$ have to be adjacent.  Let $f$ be their common endopoint and assume WLOG that $x\in F_1$ and $y\in F_2$.  
We can then estimate the difference we want to consider by 
$$
|\varphi(x)-\varphi(y)|\leq |\varphi(x)-\varphi(f)|+|\varphi(f)-\varphi(y)|,
$$
where now $x,f\in F_1$ and $f,y\in F_2$. Therefore, we can apply the proof in the special case considered above separately to the intervals  $[x,f]$ (i.e.~considering $y=f$ in the previous case) and $[f,y]$ (i.e.~setting $x=f$ in the previous case) respectively (in particular, applying  Proposition~\ref{prop:appr} twice and producing 
two pairs of approximating sequences, $(x_n)_n $ and $(f_n)_n$ approximating $[x,f]$ first and $(y_n)_n $ and $(f'_n)_n$ approximating $[f,y]$ then). The arguments of the previous step then give that there exists constants $C_1:= C_0(x,f)$ and $C_2:= C_0(f,y)$ such that
\begin{equation}\label{eq:halfintervalsestimates}
|\varphi(x)-\varphi(f)|\leq C_1 \lambda_1^{k(x,f)}, \qquad |\varphi(f)-\varphi(y)|\leq C_2 \lambda_1^{k(f,y)}.
\end{equation}
We now claim that, by construction,  $k(x,f)=k(f,y)=k(x,y)=k_0$. To see this, on one hand, by construction $[x,f]\subset F_1$ and  $[f,y]\subset F_2$ where $F_1,F_2$ are floors of $\mathcal{P}_{k_0-1}$, so both $k(x,f)$ and $k(f,y)$ are at least $k_0$. On the other hand, since $[x,y]$ contains at least a full floor of $\mathcal{P}_{k_0}$, but $f$ is by definition a point of the partition $\mathcal{P}_{k_0-1}$ and, therefore, also of the partition $\mathcal{P}_{k_0}$, there are actually \emph{two} full floors of $\mathcal{P}_{k_0}$ with common endopoint $f$ both contained in $[x,y]$. Thus, one of them is contained in $[x,f]$ and the other in $[f,y]$, which shows that $k(x,f)$ and $k(f,y)$ are at most $k_0$. This proves the claim.
 
Therefore,  both estimates  in \eqref{eq:halfintervalsestimates} have the same exponent $k_0$ and, summing them up,  
 we get the desired estimate (with $C:=2\max \{ C_1, C_2\}$). This concludes the proof.
\end{proof}

We can now exploit the Outline and the results proven in this section to conclude the proof of Theorem~\ref{thmb}.
\begin{proof}[Proof of Theorem~\ref{thmb}]
Assume that $T, T_0$ are as in the assumptions. 
Combining Lemma~\ref{lemma:denom} and Proposition~\ref{prop:numerator}, we can conclude, as explained in the Outline in \S~\ref{sec:strategy}, that $\varphi$ is H{\"o}lder with $\alpha=:\log \lambda_2/\log \lambda_1\in (0,1]$. This in turns implies (see again the arguments in the  Outline in \S~\ref{sec:strategy}) that 
\end{proof}

\subsection{A quantitative refinement of Theorem~\ref{thmb}.}\label{sec:quantitative}
In this last paragraph, we record a more precise result, which provides a quantitative version of Theorem \ref{thmb} and  may prove useful  for future applications, in particular to show that conjugacy classes are (locally) smooth submanifolds. 

\begin{thm}
\label{thm:quantitative}
For any $d\geq 2$, for a.e.~IET $T_0$ in $\mathcal{X}^r_d$ the following holds. Assume that $T$ is a GIET in  $\mathcal{X}^r_d$, $r\geq 3$ whose the orbit  $(\mathcal{R}^m(T))_{m\in\mathbb{N}}$ of $T$ under renormalisation converges exponentially fast, in the $\mathcal{C}^2$ distance, to the subspace $\mathcal{I}_d$ of (standard) IETs, i.e.
$$
d_{\mathcal{C}^2} (\mathcal{R}^m(T),  \mathcal{I}_d)\leq C(T) \rho^m, \qquad \, \forall m \in \N,
$$
for some $C(T)>0$ and $0<\rho<1$, then there exists $0<\alpha(T_0)<1$  such  that the conjugacy $\varphi$ between $T$ and $T_0$ is actually a diffeomorphism of $[0,1]$ of class $\mathcal{C}^{1+\alpha}$, and there exists a constant $D(T_0) > 0$ such that 
$$d_{\mathcal{C}^{1+\alpha}}(\varphi, \mathrm{Id}) \leq D(T_0) C(T).$$
\end{thm}
\noindent 
\noindent The proof of  Theorem \ref{thmb} we just presented yield also the proof Theorem \ref{thm:quantitative} simply by recording the quantitative dependence of constants. Notice that the exponent $\alpha$ depend in principle on $T_0$ and $\rho$, but in practice, when one is able to prove such convergence of renormalisation, the exponent $\rho$ depends on $T_0$, so to all practical purpose a double-dependence of $\alpha$ on both $T_0$ and $\rho$ is redundant.


\subsection{Deduction of the other results.}\label{sec:final}
In this final short section, we briefly summarize how to deduce from Theorem~\ref{thmb} that we just proved and the results in \cite{SU} the other results stated in the introduction (i.e.~Theorem~\ref{thma} and Corollary~\ref{cor:foliationsrigidity}).
 We begin with~Theorem~\ref{thma}. 
\begin{proof}[Proof of Theorem~\ref{thma}]
Given a GIET of class $\mathcal{C}^3$ such that $\mathcal{B}(T)$ vanishes which is topologically conjugate to a standard IET in the full measure set of $\mathcal{I}_4 \cup \mathcal{I}_5$ given by Therem~\ref{thmc}, the conjugacy is differentiable and the orbit of $T$ under renormalisation convergence exponentially and satisfies \eqref{exp:conv} by Theorem~\ref{thmc}. Therefore, the assumptions of Theorem~\ref{thmb} hold and, intersecting with the full measure set of Theorem~\ref{thmb}, we get Theorem~\ref{thma}. 
\end{proof}
Before we sketch how to deduce the result on regularity of foliations rigidity (i.e.~Corollary~\ref{cor:foliationsrigidity}), let us first  briefly recall some basic definitions and refer the reader to \cite{SU} for more details (see in particular Section $6$ \S~6.3 of \cite{SU} formore  background concerning foliations, their regularity and their holonomies).

\medskip \noindent {\it Foliations and regularity.}
Let $S$ be a closed orientable smooth surface. We consider foliations on $S$ with a finite number of singularities, and we further ask that those singularities are of (possibly degenerate) \emph{saddle type}. When the saddle is simple, this means that, locally (i.e.~in a neighbourhood of $p$) there are charts for which the  {\it topological} model of the foliation is given by the level sets in $\R^2$ of the function $(x,y) \longmapsto xy$ around $0$; saddles with $2k$-prongs are locally modelled on the foliations given by level sets of the function $(x,y)\longmapsto \Im m ((x+i y )^{k})$. Let us denote by $\mathcal{F}$ the singular foliation on $S_g$ and by $Sing_\mathcal{F} \subset S$ be the finite set of (saddle-like) singular points of $\mathcal{F}$. When $S$ has genus two, by Poincar{\'e}-Hopf theorem either $Sing_\mathcal{F} $ consists of two simple saddles (each with $4$ prongs), or it consists of only one degenerate saddle with $6$ prongs ($k=3$). 
 Following Levitt \cite{Le:dec}, 
we say that the foliation $\mathcal{F}$ is of class $\mathcal{C}^r$ iff: 
\begin{enumerate}
\item[(r1)] the  leaves of $\mathcal{F}$ in $S_g \setminus Sing_\mathcal{F}$ are locally embedded $\mathcal{C}^r$-curves;
\item[(r2)] for any two smooth open transverse arcs $I$ and $J$ which are joined by leaves of $\mathcal{F}$, the holonomy map $I \longrightarrow J$ is a $\mathcal{C}^r$ diffeomorphism on its image and extends to the boundary of $I$ to a $\mathcal{C}^r$-diffeomorphism.
\end{enumerate} 
We stress that not every $\mathcal{C}^r$  (or even \emph{smooth}) vector field  gives rise to a $\mathcal{C}^r$ (or smooth) foliation, see Appendix $A.4$ of \cite{SU}.
The obstruction for a foliation to be $\mathcal{C}^r$-smooth (in the sense above) can be encoded through \emph{holonomies around singular points} (see \cite{SU} for definitions and details):  foliations are smooth in this sense  
as long as the holomomy around each singularity vanishes. Notice however that foliations defined by $\mathcal{C}^r$ vector fields with \emph{non-degenerate} critical points (which equivalently implies   exactly that the leaves in a neighbourhood of each critical points are locally defined by $\mathcal{C}^r$ \emph{Morse} functions) are automatically of class $\mathcal{C}^{r}$  in the above sense. 

\medskip \noindent {\it Measured foliations and Katok measure class.} A special and much studied class of foliations is given by \emph{measured foliations}, namely foliations endowed with an absolutely continuous transverse measure invariant by holonomy. These  include foliations whose leaves are  trajectories of linear flows on translation surfaces and more generally foliations given by  by a smooth, closed $1$-form $\eta$ such that $\eta$ vanishes at only finitely many points which are (multi)saddles (described by level sets of smooth functions near a zero of finite multiplicity).  
Katok showed in \cite{Ka:inv} that (orientable) foliations with only Morse saddles and a non-atomic invariant measure is locally determined (up to smooth isotopy fixing the set of singular points $Sing_\mathcal{F}$) by a (relative) cohomology class $\omega$ in  $\mathrm{H}^1(S_g, Sing_\mathcal{F}, \mathbb{R})$, known in the literature as 
\emph{Katok fundamental class}. 
  We can therefore endow the space of orientable measured foliations with fixed Morse-type singularities, up to isotopy, with the affine structure of $\mathrm{H}^1(S_g, Sing_\mathcal{F}, \mathbb{R}) =  \mathbb{R}^d$. 
The Lebesgue measure on $\mathbb{R}^d$ then induces a \emph{measure class} (i.e.~a notion of measure zero sets). In the statement of  Corollary~\ref{cor:foliationsrigidity}, full measure means that the complement has measure zero in this sense. 


\medskip \noindent {\it  Foliations $\mathcal{C}^{1+\alpha}$-rigidity in genus two.} We can now summarize the proof of Corollary~\ref{cor:foliationsrigidity}. Recall that the existence of a topological conjugacy as well as the notion of regularity of conjugacies between (minimal) foliations is equivalent to the existence of a topological conjugacy and the notion of regularity for their Poincar{\'e} maps.

\begin{proof}[Proof of Corollary~\ref{cor:foliationsrigidity}]  
Given a minimal orientable measured foliation on a surface $S$ of genus two, one can choose a transversals $I\subset X$ and coordinates such that 
the  corresponding first return map to $I$ are (standard) IET with $d=4$ or $d=5$ exchanged subintervals and irreducible combinatorial datum. We will call such transversals \emph{normal}\footnote{In general, the Poincar{\'e} map on a transversal can be an IET with $d, d+1 $ or $d+2$ continuity intervals, where $d=2g+\kappa-1$,  $g$ is the genus of $S$ and $\kappa$ the cardinality of $\Sing_{\mathcal{F}}$. Normal transversals (which minimize the number of exchanged intervals) are chosen so that the enpoints of $I$ belong to saddle separatrices.}.   
Consider the full measure class of IETs with $d=4$ or $d=5$ given Theorem~A. Let us say that an orientable measured foliations on $S$ belongs to $\mathcal{D}$ if there exists a normal transversal such that the IET which arise as Poincar{\'e} section  belongs to this class of IETs. Since a result holds for a full measure set of foliations with fundamental class in  $\mathrm{H}^1(S, Sing_\mathcal{F}, \mathbb{R}) =  \mathbb{R}^d$ if it holds for almost every IET with $d$ continuity interval and irreducible combinatorial datum (see e.g.~\cite{Ul:abs}), considering $d=4$ or $d=5$ (which correspond to foliations on $S$ of genus two with one or two Morse singularities respectively), we deduce from Theorem~A that $\mathcal{D}$ has full measure.  

If $\mathcal{F}$ is a foliation of class $\mathcal{C}^3$ which by assumption $(i)$ is topologically conjugated to a foliation $\mathcal{F}_0$ in $\mathcal{D}$,   by definition there exists a transverals $I\subset S$ and coordinates such that the Poincar{\'e} map of $\mathcal{F}$ to $I$
is a GIET of class $\mathcal{C}^3$ which is topologically conjugated to an IET $T_0$ which satisfy the conclusion of Theorem~A. Furthermore, the assumption $(ii)$ that the holonomies are zero imply that the boundary $\mathcal{B}(T)=0$ (see  \S~6.3 of \cite{SU} for the proof). 
Thus, we can apply Theorem~A to conclude that there exists $0<\alpha<1 $ such that $T$ and $T_0$, and hence equivalently $\mathcal{F}$ and $\mathcal{F}_0$, are 
 $\mathcal{C}^{1+\alpha}$ conjugated.
\end{proof}

\subsubsection{Acknowledgements} C.U.~is partially supported by the {\it Swiss National Science Foundation} through Grant $200021\_188617/1$. {\color{blue}}

\bibliographystyle{plain}

\end{document}